\documentclass{amsart}

\usepackage{amssymb,esint,color}
\usepackage[colorlinks=true,urlcolor=blue, citecolor=red,linkcolor=blue,
linktocpage,pdfpagelabels, bookmarksnumbered,bookmarksopen]{hyperref}
\usepackage[hyperpageref]{backref}

\usepackage{verbatim}

\newtheorem{teo}{Theorem}[section]
\newtheorem{lm}[teo]{Lemma}
\newtheorem{prop}[teo]{Proposition}
\newtheorem{coro}[teo]{Corollary}
\theoremstyle{definition}
\newtheorem{oss}[teo]{Remark}

\newtheorem*{ack}{Acknowledgments}

\title[Orthotropic with nonstandard growth]{Lipschitz regularity for\\ orthotropic functionals\\ with nonstandard growth conditions}

\author[Bousquet]{Pierre Bousquet}

\author[Brasco]{Lorenzo Brasco}

\address[P. Bousquet]{Institut de Math\'ematiques de Toulouse, CNRS UMR 5219
\newline\indent Universit\'e de Toulouse
\newline\indent F-31062 Toulouse Cedex 9, France.}
\email{pierre.bousquet@math.univ-toulouse.fr}

\address[L.\ Brasco]{Dipartimento di Matematica e Informatica
	\newline\indent
	Universit\`a degli Studi di Ferrara
	\newline\indent
	Via Machiavelli 35, 44121 Ferrara, Italy}
\email{lorenzo.brasco@unife.it}


\numberwithin{equation}{section}

\begin{document}

\begin{abstract}
We consider a model convex functional with orthotropic structure and super-quadratic nonstandard growth conditions. We prove that bounded local minimizers are locally Lipschitz, with no restrictions on the ratio between the highest and the lowest growth rate.
\end{abstract}

\maketitle

\begin{center}
\begin{minipage}{11cm}
\small
\tableofcontents
\end{minipage}
\end{center}

\section{Introduction}

\subsection{Overview}

We pursue our study of the gradient regularity for local minimizers of functionals from the Calculus of Variations, having a structure that we called {\it orthotropic}. We refer to our previous contributions \cite{BB, BBJ, BBLV, BC} and \cite{BLPV}, for an introduction to the subject.
\par
More precisely, we want to expand the investigation carried on in \cite{BLPV}, by studying functionals of the form
\[
\int f(\nabla u)\,dx,\qquad f:\mathbb{R}^N\to \mathbb{R} \mbox{ convex},
\]
which couple the following two features
\[
\mbox{\it orthotropic structure} \qquad \mbox{ and }\qquad \mbox{\it nonstandard growth conditions}.
\]
The first one means that we require
\[
f(z)=\sum_{i=1}^N f_i(z_i),\qquad \mbox{ with } f_i:\mathbb{R}\to\mathbb{R} \mbox{ convex},
\]
while the second one means that 
\[
|z|^p-1\lesssim f(z)\lesssim |z|^q+1,\qquad \mbox{ with } 1<p<q.
\]
As we will see, these two features give rise to one of the most challenging type of functionals, at least if one is interested in higher order regularity of local minimizers, i.e. regularity of their gradients.
\par
Let us be more specific on the type of functionals we want to study.
We take a vector $\mathbf{p}=(p_1,\dots,p_N)$ with $2\le p_1\le \dots\le p_N$. Let $\Omega\subset\mathbb{R}^N$ be an open set. For every $u\in W^{1,\mathbf{p}}_{\rm loc}(\Omega)$ and every $\Omega'\Subset\Omega$, we consider the {\it orthotropic functional with nonstandard growth}
\[
\mathfrak{F}_{\mathbf{p}}(u,\Omega')=\sum_{i=1}^N \frac{1}{p_i}\,\int_{\Omega'} |u_{x_i}|^{p_i}\,dx.
\]
We say that $u\in W^{1,\mathbf{p}}_{\rm loc}(\Omega)$ is a {\it local minimizer of $\mathfrak{F}_p$} if
\[
\mathfrak{F}_{\mathbf{p}}(u,\Omega')\le \mathfrak{F}_{\mathbf{p}}(v,\Omega'),\qquad \mbox{ for every } v-u\in W^{1,\mathbf{p}}_0(\Omega')\quad \mbox{ and every }\Omega'\Subset\Omega.
\] 
Here $W^{1,\mathbf{p}}$ and $W^{1,\mathbf{p}}_0$ are the classical anisotropic Sobolev spaces, defined for an open set $E\subset\mathbb{R}^N$ by
\[
W^{1,\mathbf{p}}(E)=\{u\in L^1(E)\, :\, u_{x_i}\in L^{p_i}(E),\, i=1,\dots,N\},
\]
and
\[
W^{1,\mathbf{p}}_0(E)=W^{1,\mathbf{p}}(E)\cap W^{1,1}_0(E).
\]
It is easy to see that a local minimizer of $\mathfrak{F}_p$ is a local weak solution of the following quasilinear equation with orthotropic structure
\begin{equation}
\label{orthopi}
\sum_{i=1}^N \Big(|u_{x_i}|^{p_i-2}\,u_{x_i}\Big)_{x_i}=0.
\end{equation}
It is well-known that local minimizers of functionals like $\mathfrak{F}_\mathbf{p}$ above can be {\it unbounded} if the ratio
\[
\frac{p_N}{p_1},
\]
is too large, see the celebrated counter-examples by Giaquinta \cite{Gia} and Marcellini \cite{Ma91, Ma89, MaContro} (see also Hong's paper \cite{Ho}). In Western countries, the regularity theory for {\it non degenerate} functionals with nonstandard growth was initiated in the seminal papers \cite{Ma91,Ma89} by Marcellini. For {\it strongly degenerate functionals}, including the orthotropic functional with nonstandard growth $\mathfrak{F}_\mathbf{p}$ introduced above, the question has been addressed in the Russian litterature, see for example the papers \cite{Ko} by Kolod\={\i}\u{\i}, \cite{Kor} by Koralev and \cite{UU} by Uralt'seva and Urdaletova. 
\par
However, in spite of a large number of papers and contributions on the subject, a satisfactory gradient regularity theory for these problems is still missing. Some higher integrability results for the gradient have been obtained for example in  \cite[Theorem 2.1]{ELM1} and \cite[Theorem 5]{ELM}. In any case, we point out that even the case of basic regularity (i.e. $C^{0,\alpha}$ estimates and Harnack inequalities) is still not completely well-understood, we refer to the recent paper \cite{BDP} and the references therein.

\subsection{Main result}
In this paper, we are going to prove that {\it bounded} local minimizers of our orthotropic functional $\mathfrak{F}_\mathbf{p}$ are locally Lipschitz continuous. We point out that {\bf no upper bounds on the ratio}
\[
\frac{p_N}{p_1},
\]
{\bf are needed} for the result to hold.
\begin{teo}
\label{teo:lipschitz}
Let  $\mathbf{p}=(p_1,\dots,p_N)$ be such that $2\le p_1\le \dots\le p_N$.
Let $U\in W^{1,\mathbf{p}}_{\rm loc}(\Omega)$ be a local minimizer of $\mathfrak{F}_{\mathbf{p}}$ such that 
\[
U \in L^\infty_{\rm loc}(\Omega).
\]
Then $\nabla U \in L^\infty_{\rm loc}(\Omega)$.
\end{teo}

\begin{oss}[$L^\infty$ assumption]
Sharp conditions in order to get $U\in L^\infty_{\rm loc}$ can be found in \cite[Theorem 1]{FS2} by Fusco and Sbordone, see also \cite[Theorem 3.1]{FS} and the papers \cite{CMM2, CMM} by Cupini, Marcellini and Mascolo for the case of more general functionals. Pioneering results are due to Kolod\={\i}\u{\i}, see \cite{Ko}. We also mention the recent paper \cite{DGV} by DiBenedetto, Gianazza and Vespri, where some precise a priori $L^\infty$ estimates on the solution are proved, see Section 6 there. 
\end{oss}

\begin{oss}[Comparison with previous results]
Some particular cases of our Theorem \ref{teo:lipschitz} can be traced back in the literature. We try to give a complete picture of the subject. 
\par
The first one is \cite[Theorem 1]{UU} by Uralt'seva and Urdaletova, which proves local Lipschitz regularity for bounded minimizers, under the restrictions
\[
p_1\ge 4\qquad\mbox{ and }\qquad \frac{p_N}{p_1}<2.
\]
The method of proof of \cite{UU} is completely different from ours and is based on the so-called {\it Bernstein's technique}. We refer to \cite{BBJ} for a detailed description of their proof.
\par
More recently, Theorem \ref{teo:lipschitz} has been proved in the two-dimensional case $N=2$ by the second author in collaboration with Leone, Pisante and Verde, see \cite[Theorem 1.4]{BLPV}. In this case as well, the proof is different from the one we give here, the former being based on a two-dimensional trick introduced in \cite[Theorem A]{BBJ}. Still in dimension $N=2$, Lindqvist and Ricciotti in \cite[Theorem 1.2]{LR} proved $C^1$ regularity for solutions of \eqref{orthopi}, by extending to the case of nonstandard growth conditions a result of the authors, see \cite[Main Theorem]{BB}. 
\par 
In the standard growth case, i.e. when $p_1=\dots=p_N=p$, local Lipschitz regularity has been obtained in \cite[Theorem 1.1]{BBLV}. As we will explain later, the result of \cite{BBLV} is the true ancestor of Theorem \ref{teo:lipschitz}, since the latter is (partly) based on a generalization of the method of proof of the former.
An alternative proof, based on viscosity methods, has been given by Demengel, see \cite{De}.
\par
In \cite{De2}, the same author extended her result to cover the case $p_1<p_N$, under the assumptions
\[
p_N<p_1+1.
\]
The result of \cite[Corollary 1.2]{De2} still requires $p_1\ge 2$ and applies to {\it continuous} local minimizers.
\par
Finally, Lipschitz regularity for solutions of \eqref{orthopi} has been claimed in the abstract of \cite{BoDuMa}. However, a closer inspection of the assumptions of Theorem 1.2 there (see \cite[equation (1.2)]{BoDuMa}) shows that their result does not cover the case of \eqref{orthopi}.
\end{oss}
\begin{oss}[A paper by Lieberman]
The  reader who is familiar with this subject may observe that {\it apparently} our Theorem \ref{teo:lipschitz} is already contained in Lieberman's paper \cite{Li}. Indeed, \cite[Example 1, page 794]{Li} deals with exactly the same result for bounded minimizers, by even dropping the requirement $p_1\ge 2$. However,  Lieberman's proof seems to be affected by a crucial flaw. This is a delicate issue, thus we prefer to explain in a clean way the doubtful point of \cite{Li}. 
\par
We first recall that the proof by Lieberman is inspired by Simon's paper \cite{Si}, dealing with $L^\infty$ gradient estimates for solutions of non-uniformly elliptic equations. One of the crucial tool used by Simon is a generalized version of the Sobolev inequality for functions on manifolds. This is a celebrated result by Michael and Simon himself \cite[Theorem 2.1]{MS}, which in turn generalizes the idea of the cornerstone paper \cite{BDM} by Bombieri, De Giorgi and Miranda on the {\it minimal surface equation}. 
\par
The idea of \cite{Li} is to enlarge the space dimension and identify the set $\Omega$ with the flat $N-$dimensional submanifold $\mathcal{M}:=\Omega\times\{(0,\dots,0)\}$ contained in $\mathbb{R}^{2N-1}$. Then the author introduces:
\begin{itemize} 
\item a suitable gradient operator 
\[
\varphi\mapsto \delta\varphi:=\left(\sum_{j=1}^{2N-1}\gamma^{1,j}\,\varphi_{x_j},\dots,\sum_{j=1}^{2N-1}\gamma^{2N-1,j}\,\varphi_{x_j}\right).
\] 
Here $\gamma=(\gamma^{i,j})$ is a measurable map with values into the set of positive definite symmetric $(2N-1)\times(2N-1)$ matrices;
\vskip.2cm
\item a suitable nonnegative measure $\mu$ defined on sets of the form $\mathcal{M}\cap E$ for all Borel sets $E\subset\mathbb{R}^{2N-1}$;
\vskip.2cm
\item a mean curvature--type operator $H=(H_1,\dots,H_N,H_{N+1},\dots,H_{2N-1})$ defined on \(\mathcal{M}\). 
\end{itemize} 
The key point of \cite[Section 4]{Li} is to apply the Sobolev--type inequality of Michael and Simon in conjunction with a Caccioppoli inequality for the gradient, in order to circumvent the strong degeneracy of the equation \eqref{orthopi}.
However, in order to apply the result by Michael and Simon, some conditions linking the three objects above are needed. Namely, the crucial condition
\begin{equation}
\label{MS}
\int_{\mathcal{M}} \Big[\delta_i\,\varphi+H_i\,\varphi\Big]\,d\mu=0,\qquad \mbox{ for every } \varphi\in C^\infty_0(U),\ \mbox{ for every }i=1,\dots,2N-1,
\end{equation}
must be verified, where $\mathcal{M}\subset U\subset\mathbb{R}^{2N-1}$ is an open set. This is condition (1.2) in \cite{MS}, which is stated to hold true within the framework of Lieberman's paper, see the proof of \cite[Proposition 2.1]{Li}. However, with the definitions of $\mu, \delta$ and $H$ taken in \cite{Li}, one can see that {\it this crucial condition fails to be verified}. Indeed, with the definitions of \cite[Proposition 2.1]{Li}, for $N+1\le i\le 2N-1$, it holds
\[
\gamma^{i,i}=1\quad \mbox{ and }\quad \gamma^{i,j}=0\mbox{ for } j\not=i, \qquad \mbox{ thus }\quad \delta_i\,\varphi=\varphi_{x_i},
\]
and
\[
H_i=0,
\]
while $\mu$ coincides with the $N-$dimensional Lebesgue measure on $\Omega$. Thus condition \eqref{MS} for $N+1\le i\le 2N-1$ becomes
\[
\int_\Omega \varphi_{x_i}(x,0,\dots,0)\,dx=0,\qquad \mbox{ for every } \varphi\in C^\infty_0(U),
\]
which in general {\it is false}. Thus the proof of \cite[Proposition 2.1]{Li} does not appear to be correct, leaving in doubt the whole proof of \cite[Lemma 4.1]{Li}, which contains the $L^\infty$ gradient estimate. 
\end{oss}

\subsection{Structure of the proof}

The proof of Theorem \ref{teo:lipschitz} is quite involved, thus we prefer spending a large part of this introduction in order to neatly introduce the main ideas and novelties. 
\par
As usual when dealing with higher order regularity, the first issue to be tackled is that the minimizer $U$ lacks the smoothness needed to perform all the necessary manipulations. However, this is a minor issue, which can be easily fixed by approximating our local minimizer $U$ with solutions $u_\varepsilon$ of uniformly elliptic problems, see Section \ref{sec:preliminaries}. The solutions $u_\varepsilon$ are as smooth as needed (basically, $C^2$ regularity is enough) and they converge to our original local minimizer $U$, as the small regularization parameter $\varepsilon>0$ converges to $0$. Thus it is sufficient to prove ``good'' a priori estimates on $u_\varepsilon$ which are stable when $\varepsilon$ goes to $0$. 
\par
For this reason, in the rest of this subsection we will pretend that our minimizer $U$ is smooth and explain how to get the needed a priori estimates.
\vskip.2cm\noindent
The building blocks of Theorem \ref{teo:lipschitz} are the following two estimates:
\begin{itemize}
\item[{\bf A.}] a local $L^\infty-L^\gamma$ a priori estimate on the gradient, i.e. an estimate of the type
\begin{equation}
\label{linftyapriori}
\|\nabla U\|_{L^\infty(B_r)}\le C\, \left(\int_{B_R} |\nabla U|^\gamma\,dx\right)^\frac{\Theta}{\gamma},
\end{equation}
where $\gamma\ge p_N+2$ and $\Theta>1$ are two suitable exponents. This is the content of Proposition \ref{prop:a_priori_estimate};
\vskip.2cm
\item[{\bf B.}] a local higher integrability estimate {\it of arbitrary order} on the gradient, i.e. an estimate of the type
\[
\int_{B_R} |\nabla U|^q\,dx\le C_q,
\]
where $1<q<+\infty$ is arbitrary and $C_q>0$ is a constant depending only on $q$, the data of the problem {\it and the local $L^\infty$ norm of} $U$. This is proved in Proposition \ref{prop:pierre}.
\end{itemize}  
It is straightforward to see that once {\bf A.} and {\bf B.} are established, then our main result easily follows. We explain how to get both of them:
\vskip.2cm\noindent
\begin{itemize}
\item in order to obtain {\bf A.} we employ the same method that we successfully applied in \cite[Theorem 1.1]{BBLV}, for the standard growth case $p_1=\dots=p_N=p$. This is based on a new class of Caccioppoli-type inequalities for $\nabla U$, which have been first introduced by the two authors in \cite{BB} and then generalized and exploited in its full generality in \cite{BBLV}.
\par
In a nutshell, the idea is to take the equation satisfied by $U$
\[
\sum_{i=1}^N \Big(|U_{x_i}|^{p_i-2}\,U_{x_i}\Big)_{x_i}=0,
\]
differentiate it with respect to $x_j$ and then insert weird test functions of the form
\[
\Phi(U_{x_j})\,\Psi(U_{x_k}),
\]
with $k,j\in\{1,\dots,N\}$. With these new Caccioppoli-type inequalities at hand, we can follow the same scheme as in \cite[Proposition 5.1]{BBLV} and obtain \eqref{linftyapriori}.
\par
We point out that, apart from a number of technical complications linked to the fact that $p_1\not=p_N$, in the present setting there is a crucial difference with the case treated in \cite{BBLV}. Indeed, after a Moser--type iteration, there we obtained an a priori estimate of the type
\begin{equation}
\label{puguali}
\|\nabla U\|_{L^\infty(B_r)}\le C\, \left(\int_{B_R} |\nabla U|^{p+2}\,dx\right)^\frac{1}{p+2}.
\end{equation}
Apparently, in that case as well we needed the further higher integrability information $\nabla U\in L^{p+2}_{\rm loc}$. However, thanks to the homogeneity of the estimate \eqref{puguali}, one can use a standard interpolation trick (see the Step 4 of the proof of \cite[Proposition 5.1]{BBLV}) and upgrade \eqref{puguali} to the following
\[
\|\nabla U\|_{L^\infty(B_r)}\le C\, \left(\int_{B_R} |\nabla U|^{p}\,dx\right)^\frac{1}{p}.
\]
This does not require any prior integrability information on $\nabla U$ beyond the natural growth exponent. Thus in the standard growth case, we are dispensed with point {\bf B.}, i.e. point {\bf A.} is enough to conclude.
On the contrary, in our case, the same trick does not apply to estimate \eqref{linftyapriori}, because of the presence of the exponent $\Theta>1$. For this reason we need a higher integrability information on $\nabla U$.
\par
In order to have a better understanding of the proof described above, we refer the interested reader to  the Introduction of \cite{BBLV}, where the whole strategy for point {\bf A.} is explained in details;
\vskip.2cm\noindent
\item part {\bf B.} is the really involved point of the whole proof. At first, we point out that we do not have a good control on the exponent $\gamma$ appearing in \eqref{linftyapriori} (unless some restrictions on the ratio $p_N/p_1$ are imposed). For this reason, we need to gain as much integrability on $\nabla U$ as possible. This is a classical subject in the regularity theory for functionals with nonstandard growth conditions, i.e. integrability gain on the gradients of minimizers. 
\par
Since in general minimizers of this kind of functionals may be very irregular when $p_1$ and $p_N$ are too far apart, usually one needs to impose some restrictions on the ratio $p_N/p_1$ to get some regularity. These restriction are typically of the type
\[
\frac{p_N}{p_1}< c_N,\quad \mbox{ for some constant $c_N>0$ such that }\quad \lim_{N\to\infty} c_N=1.
\]
If one further supposes {\it local minimizers to be bounded}, then the previous restriction can be relaxed to conditions of the type 
\[
\frac{p_N}{p_1}< C\qquad \mbox{ or }\qquad p_N<p_1+C,
\] 
with a {\it universal} constant $C>0$ . In any case, to the best of our knowledge all the results appearing in the literature require some upper bound on the ratio $p_N/p_1$. More precisely, all the results except one: in the very interesting paper \cite{BFZ} by Bildhauer, Fuchs and Zhong, the authors consider a functional with nonstandard growth of the type
\begin{equation}
\label{BFZfunctional}
(u,\Omega')\mapsto \int_{\Omega'} \left(\sum_{i=1}^{N-1} |u_{x_i}|^2\right)^\frac{p_1}{2}\,dx+\int_{\Omega'} |u_{x_N}|^{p_2}\,dx,\qquad \mbox{ with } p_1\le p_2,
\end{equation}
and prove that any local minimizer $u\in L^\infty_{\rm loc}$ is such that $\nabla u\in L^q_{\rm loc}$ for every $1<q<+\infty$, {\it no matter how large the ratio $p_2/p_1$ is}, see \cite[Theorem 1.1]{BFZ}. The idea of \cite{BFZ} is partially inspired from Choe's result  \cite[Theorem 3]{Ch}, which in turn  seems to find its roots in DiBenedetto's paper \cite{DiB} (see \cite[Proposition 3.1]{DiB}). It relies on a suitable integration by parts in conjunction with the Caccioppoli inequality for $\nabla u$. For functionals as in \eqref{BFZfunctional}, this leads   to an iterative scheme of the type
\[
``\mbox{gain of integrability on $u_{x_N}$}''\quad \Longrightarrow \quad ``\mbox{gain of integrability on $(u_{x_1},\dots,u_{x_{N-1}})$}''
\] 
and viceversa
\[
``\mbox{gain of integrability on $(u_{x_1},\dots,u_{x_{N-1}})$}''\quad \Longrightarrow \quad ``\mbox{gain of integrability on $u_{x_N}$}''.
\] 
By means of a {\it doubly recursive scheme} which is quite difficult to handle, \cite{BFZ} exploits the full power of the above approach to avoid any unnecessary restriction on the exponents. In \cite{Ch} instead, the gain of integrability was extremely simplified, at the price of taking the assumption 
\[
p_N<p_1+1.
\] 
Incidentally, we point out that this is the same assumption as in the aforementioned paper \cite{De2}, which uses however different techniques.
\par
We will try to detail the main difficulties of this method in a while. Before this, we point out that \eqref{BFZfunctional} is only concerned with two growth exponents. Moreover, the type of degeneracy of the functional \eqref{BFZfunctional} is much lighter than that of our functional $\mathfrak{F}_p$. For these reasons, even if our strategy is greatly inspired by that of \cite{BFZ}, all the estimates have to be recast and the resulting iterative scheme becomes of far reaching complexity.
\par
We now come to explain such an iterative scheme: by proceeding as in \cite{Ch} and \cite{BFZ}, we get an estimate of the type (see Proposition \ref{prop:BFZ})
\begin{equation}
\label{BFZintro}
\begin{split}
\int_{B_{r_0}}  | U_{x_k}|^{p_k+2+\alpha}\,dx \leq C+C\,\sum_{i\not= k}\int_{B_{R_0}} |U_{x_i}|^{\frac{p_i-2}{p_k}\,(p_k+2+\alpha)}\,dx,\quad \mbox{ for } k=1,\dots,N,
\end{split}
\end{equation}
where $C>0$ depends on the data and on the local $L^\infty$ norm of $U$, as well. Here the free parameter 
$\alpha \ge 0$ has to be carefully chosen, in order to improve the gradient summability.
We observe that, technically speaking, {\it this scheme is not of Moser-type}. Indeed, the key point of \eqref{BFZintro} is that it entails estimates on a fixed component $U_{x_j}$, in terms of all the others.  
\vskip.2cm\noindent
$\bullet$ \boxed{\mbox{\tt First step}} We start by using \eqref{BFZintro} as follows: we take $k=N$ in \eqref{BFZintro} and choose $\alpha\ge 0$ in such a way that
\[
\frac{p_i-2}{p_N}\,(p_N+2+\alpha)\le p_i,\qquad \mbox{ for } i=1,\dots,N-1.
\]
It is possible to make such a choice {\it without imposing restrictions on} $p_N/p_1$, the optimal choice being\footnote{For ease of presentation, in what follows we assume that $p_i>2$ for every $i=1,\dots,N$.}
\[
p_N+2+\alpha^{(0)}_N=p_N\, \min_{1\le i\le N-1} \frac{p_i}{p_i-2}=:p_N\,q_{N-1}.
\] 
This permits to upgrade the integrability of $U_{x_N}$ to $L^{p_N\,q_{N-1}}_{\rm loc}$. This is the end of the first step.
\vskip.2cm\noindent
$\bullet$ \boxed{\mbox{\tt Second step}}
Once we gain this property on $U_{x_N}$, we shift to $U_{x_{N-1}}$: we take $k=N-1$ in \eqref{BFZintro}, that we write in the following form 
\[
\begin{split}
\int_{B_{r_0}}  | U_{x_k}|^{p_k+2+\alpha}\,dx \leq C&+C\,\sum_{i=1}^{N-2}\int_{B_{R_0}} |U_{x_i}|^{\frac{p_i-2}{p_k}\,(p_k+2+\alpha)}\,dx\\
&+C\,\int_{B_{R_0}} |U_{x_N}|^{\frac{p_N-2}{p_k}\,(p_k+2+\alpha)}\,dx.
\end{split}
\]
Then by using that
\[
U_{x_i}\in L^{p_i}_{\rm loc},\ \mbox{ for } i=1,\dots,N-2 \qquad \mbox{ and }\qquad U_{x_N}\in L^{p_N\,q_{N-1}}_{\rm loc},
\] 
we choose $\alpha$ in such a way that
\[
\left\{\begin{array}{cccl}
\dfrac{p_i-2}{p_{N-1}}\,(p_{N-1}+2+\alpha)&\le& p_i,& \mbox{ for } i=1,\dots,N-2,\\
&&&\\
\dfrac{p_N-2}{p_{N-1}}\,(p_{N-1}+2+\alpha)&\le& p_N\,q_{N-1}&
\end{array}
\right.
\]
If we set as above
\[
q_i=\frac{p_i}{p_i-2},
\]
the optimal choice is now
\[
p_{N-1}+2+\alpha^{(0)}_{N-1}=p_{N-1}\,\min \left\{q_N\,q_{N-1},\, \min_{1\le i\le N-2} q_i\right\}.
\]
However, {\it this is not the end} of the second step. Indeed, rather than applying \eqref{BFZintro} directly to the other components \(U_{x_{N-2}}, \dots , U_{x_{1}}\) as above, we come back to \(U_{x_N}\). 
\par
More precisely, we apply \eqref{BFZintro} to \(U_{x_N}\) taking into account the new information on \(U_{x_{N-1}}\). This gives higher integrability for  \(U_{x_N}\).  We next apply alternatively \eqref{BFZintro} to \(U_{x_N-1}\) and \(U_{x_{N}}\), taking into account the higher integrability gain at each step.  After a  finite number of iterations, it can be established that
\[
U_{x_N}\in L^{p_N\, q_{N-2}}_{\rm loc}\qquad \mbox{ and }\qquad U_{x_{N-1}}\in L^{p_{N-1}\, q_{N-2}}_{\rm loc}.
\] 
This is the end of the second step. 
\vskip.2cm\noindent
$\bullet$ \boxed{\mbox{\tt $j-$th step}}  (for $2\le j\le N-1$) When we land on this step, we have iteratively acquired the following knowledge
\[
U_{x_i}\in L^{p_i}_{\rm loc},\ \mbox{ for } i=1,\dots,N-j \qquad \mbox{ and }\qquad U_{x_i}\in L^{p_N\,q_{N-j+1}}_{\rm loc}, \ \mbox{ for } i=N-j+2,\dots,N.
\] 
We then start to get into play the component $U_{x_{N-j+1}}$. We take $k=N-j+1$ in \eqref{BFZintro} and we choose $\alpha$ in such a way that
\[
\left\{\begin{array}{cccl}
\dfrac{p_i-2}{p_{N-j+1}}\,(p_{N-j+1}+2+\alpha)&\le& p_i,& \mbox{ for } i=1,\dots,N-j,\\
&&&\\
\dfrac{p_i-2}{p_{N-j+1}}\,(p_{N-j+1}+2+\alpha)&\le& p_i\,q_{N-j+1},&\mbox{ for } i=N-j+2,\dots,N.
\end{array}
\right.
\]
This permits to infer that 
\[
U_{x_{N-j+1}}\in L^{p_{N-j}+2+\alpha^{(0)}_{N-j+1}}_{\rm loc},
\]
where
\[
p_{N-j+1}+2+\alpha^{(0)}_{N-j+1}=p_{N-j+1}\,\min\left\{\min_{i=1,\dots,N-j}q_i,\, q_{N-j+1}\min_{i=N-j+2,\dots,N} q_i\right\}.
\]
As illustrated in the second step, we now use this information and start to cyclically use \eqref{BFZintro} on $U_{x_N}, U_{x_{N-1}},\dots, U_{x_{N-j}}$, in order to improve their integrability. After a finite number of iterations of this algorithm, we obtain
\[
U_{x_i}\in L^{p_i\,q_{N-j+1}}_{\rm loc},\qquad \mbox{ for } i=N-j+2,\dots,N.
\]
This is the end of the $j-$th step.
\vskip.2cm\noindent
$\bullet$ \boxed{\mbox{\tt Last step}.}
We fix $q_0\ge 2$ arbitrary. We finally consider the last component $U_{x_1}$, as well. By using the starting information 
\[
U_{x_i}\in L^{p_i\,q_1},\qquad \mbox{ for } i=2,\dots,N,
\]
and
proceeding as above, we finally get
\[
U_{x_i}\in L^{p_i\, q_0}_{\textrm{loc}},\qquad \mbox{ for } 1\le i\le N. 
\]
This yields the desired conclusion.  
\vskip.2cm
The main difficulty of {\bf B.} is to prove that this algorithm does not require any restriction on the exponents \(p_i\), and that each step {\it ends up after a finite number of loops}. 

\end{itemize}

\subsection{Plan of the paper}
In Section \ref{sec:preliminaries} the reader will find the approximating scheme and all the basic material needed to understand the sequel of the paper. Section \ref{sec:caccioppoli} contains the crucial Caccioppoli-type inequalities for the gradient, needed to build up the Moser's scheme for point {\bf A.} of the strategy presented above. Then in Section \ref{sec:leerp}, we prove integral estimates for the gradient: the first one is a Caccioppoli inequality for power functions of the gradient (Proposition \ref{prop-russian}), while the second one is the self-improving scheme {\it \`a la} Bildhauer-Fuchs-Zhong (Proposition \ref{prop:BFZ}). With Sections \ref{sec:5} and \ref{sec:6}, we enter into the core of the paper: they contain the $L^\infty-L^\gamma$ gradient estimate and the higher integrability estimate for the gradient, respectively. Then in the short Section \ref{sec:7}, we eventually prove our main result.
\par
Two technical appendices conclude the paper: they contain the study of all the intricate sequences of real numbers needed in this paper.

\begin{ack}
Part of this work has been done during a visit of P. B. to Bologna \& Ferrara in February 2018 and during a visit of L. B. to Toulouse in June 2018. The latter has been financed by the ANR project ``{\it Entropies, Flots, In\'egalit\'es}'', we wish to thank Max Fathi. Hosting institutions are kindly acknowledged.
\end{ack}

\section{Preliminaries}
\label{sec:preliminaries}

We will use the same approximation scheme as in \cite[Section 2]{BBJ} and \cite[Section 5]{BLPV}.
We recall that  are interested in local minimizers of the following variational integral
\[
\mathfrak{F}_{\mathbf{p}}(u;\Omega')=\sum_{i=1}^N \frac{1}{p_i}\,\int_{\Omega'} |u_{x_i}|^{p_i}\, dx,\qquad u\in W^{1,\mathbf{p}}_{\rm loc}(\Omega),\ \Omega'\Subset\Omega,
\]
where $\mathbf{p}=(p_1,\dots,p_N)$ and $2\le p_1\le \dots\le p_N$.
In the rest of the paper, {\it we fix $U\in W^{1,\mathbf{p}}_{\rm loc}(\Omega)$ a local minimizer of} $\mathfrak{F}_\mathbf{p}$. We also fix a ball 
\[
B \Subset \Omega\quad \mbox{ such that }\quad 2\,B\Subset\Omega \mbox{ as well}.
\] 
We use the usual notation $\lambda\,B$ to denote the ball concentric with $B$, scaled by a factor $\lambda>0$. Since we have the continuous inclusion \(W^{1,\mathbf{p}}(2\,B) \subset W^{1,p_1}(2\,B)\), by Poincar\'e inequality it holds 
\[
U\in L^{p_1}(2\,B).
\]
For every $0<\varepsilon\ll 1$ and every \(x\in \overline{B}\), we set $U_\varepsilon(x)=U\ast \varrho_\varepsilon(x)$, where $\varrho_\varepsilon$ is the usual family of Friedrichs mollifiers, supported in a ball of radius $\varepsilon$ centered at the origin. We also set
\begin{equation}
\label{gepsilon}
g_{i,\varepsilon}(t)=\frac{1}{p_i}\, |t|^{p_i}+\frac{\varepsilon}{2}\, t^2,\qquad t\in\mathbb{R},\ i=1,\dots,N.
\end{equation}
Finally, we define the regularized functional
\[
\mathfrak{F}_{\mathbf{p},\varepsilon}(v;B)=\sum_{i=1}^N \int_B g_{i,\varepsilon}(v_{x_i})\, dx.
\]
The following preliminary result is standard, see \cite[Lemma 2.5 and Lemma 2.8]{BBJ}.
\begin{lm}[Basic energy estimate]
\label{lm:below}
There exists $0<\varepsilon_0<1$ such that for every $0<\varepsilon\le \varepsilon_0$, the problem
\begin{equation}
\label{approximated}
\min\left\{\mathfrak{F}_{\mathbf{p},\varepsilon}(v;B)\, :\, v-U_\varepsilon\in W^{1,\mathbf{p}}_0(B)\right\},
\end{equation}
admits a unique solution $u_\varepsilon$. 
Moreover, the following uniform estimate holds 
\[
\sum_{i=1}^N \frac{1}{p_i}\,\int_B |(u_\varepsilon)_{x_i}|^{p_i}\, dx\le \left(\sum_{i=1}^N \frac{1}{p_i}\,\int_{2\,B} |U_{x_i}|^{p_i}\,dx+\frac{\varepsilon_0}{2}\,\int_{2\,B} |\nabla U|^2\,dx\right).
\]
Finally, $u_\varepsilon\in C^{2}(B)$.
\end{lm}
\begin{proof}
The only difference with respect to \cite{BBJ} is on the uniform energy estimate, due to the nonstandard growth conditions. We show how to obtain this: it is sufficient to test the minimality of $u_\varepsilon$ against $U_\varepsilon$, this gives
\[
\begin{split}
\sum_{i=1}^N \frac{1}{p_i}\,\int_B |(u_\varepsilon)_{x_i}|^{p_i}\, dx&\le \sum_{i=1}^N \frac{1}{p_i}\,\int_{B} |(U\ast\varrho_\varepsilon)_{x_i}|^{p_i}\,dx+\frac{\varepsilon}{2}\,\int_B |\nabla (U\ast \varrho_\varepsilon)|^2\,dx\\
&\le \|\varrho_\varepsilon\|_{L^1(\mathbb{R}^N)}\,\left(\sum_{i=1}^N \frac{1}{p_i}\,\int_{2\,B} |U_{x_i}|^{p_i}\,dx+\frac{\varepsilon}{2}\,\int_{2\,B} |\nabla U|^2\,dx\right).
\end{split}
\]
By using the scaling properties of the family of mollifiers, we get the conclusion.
\end{proof}
As usual, we will also rely on the following convergence result.
\begin{lm}[Convergence to a minimizer]
\label{lm:convergence}
With the same notation as above, we have
\begin{equation}
\label{troppoforte!}
\lim_{\varepsilon\to 0} \left[\|u_{\varepsilon}-U\|_{L^{p_1}(B)}+ \sum_{i=1}^N \|(u_\varepsilon-U)_{x_i}\|_{L^{p_i}(B)}\right]=0.
\end{equation}
\end{lm}
\begin{proof}
We observe that $u_\varepsilon-U_\varepsilon\in W^{1,\mathbf{p}}_0(B)$ and the set $B$ is bounded in every direction. Thus by Poincar\'e inequality, we have
\begin{equation}
\label{poincare}
\int_B |u_\varepsilon-U_\varepsilon|^{p_i}\le C_i\,|B|^{-\frac{p_i}{N}}\,\int_B |(u_\varepsilon-U_\varepsilon)_{x_i}|^{p_i}\,dx, \qquad i=1,\dots,N,
\end{equation}
for some $C_i=C_i(N,p_i)>0$.
For $i=1$, this in turn gives
\[
\begin{split}
\|u_\varepsilon\|_{L^{p_1}(B)}&\le \|u_\varepsilon-U_\varepsilon\|_{L^{p_1}(B)}+\|U_\varepsilon\|_{L^{p_1}(B)}\\
&\le C\,\|(u_\varepsilon-U_\varepsilon)_{x_1}\|_{L^{p_1}(B)} +\|U_\varepsilon\|_{L^{p_1}(B)}\\ 
&\le C\,\|(u_\varepsilon)_{x_1}\|_{L^{p_1}(B)}+C\,\|U\|_{W^{1,p_1}(2\,B)},
\end{split}
\]
for a constant $C=C(N,p_1)>0$.
By Lemma \ref{lm:below}, the last term is uniformly bounded for $0<\varepsilon\le\varepsilon_0$.
Thus the family $\{u_\varepsilon\}_{0<\varepsilon\le \varepsilon_0}$ is bounded in $W^{1,p_1}(B)$. We can infer the weak convergence in $W^{1,p_1}(B)$ of a subsequence $\{u_{\varepsilon_k}\}_{k\in \mathbb{N}}\) to a function $u\in W^{1,p_1}(B)$. This convergence is strong in $L^{p_1}(B)$, by the Rellich-Kondra\v{s}ov Theorem. 
\par
For every $\varphi\in W^{1,\mathbf{p}}_0(B)$, we test the minimality of $u_{\varepsilon_k}$ against $\varphi+U_{\varepsilon_k}$. Thus, by lower semicontinuity of the $L^{p_i}$ norms on \(L^{p_1}(B)\), we can infer 
\begin{equation}
\label{maroc}
\begin{split}
\sum_{i=1}^N\frac{1}{p_i}\, \int_B |u_{x_i}|^{p_i}\, dx&\le\liminf_{k\to +\infty}\sum_{i=1}^N\frac{1}{p_i}\, \int_B |(u_{\varepsilon_k})_{x_i}|^{p_i}\, dx\\
&\le \lim_{k\to +\infty}\sum_{i=1}^N \frac{1}{p_i}\,\int_B |(\varphi+U_{\varepsilon_k})_{x_i}|^{p_i}\, dx+\frac{\varepsilon_k}{2}\,\int_B |\nabla \varphi+\nabla U_{\varepsilon_k}|^2\,dx\\
&=\sum_{i=1}^N \frac{1}{p_i}\,\int_B |(\varphi+U)_{x_i}|^{p_i}\, dx.
\end{split}
\end{equation}
This shows that $u_{x_i}\in L^{p_i}(B)$ for $i=1,\dots,N$ and $u$ solves 
\[
\min\left\{\mathfrak{F}_{\mathbf{p}}(v;B)\, :\, v-U\in W^{1,\mathbf{p}}_0(B)\right\}.
\]
By strict convexity of the functional $\mathfrak{F}_\mathbf{p}$, we thus obtain $u=U$.
\par
We can now take $\varphi\equiv 0$ in \eqref{maroc}. Since we know that $u=U$, we have equality everywhere in \eqref{maroc}, thus in particular
\begin{equation}
\label{norms}
\lim_{k\to +\infty} \sum_{i=1}^N \frac{1}{p_i}\,\int_B \left|(u_{\varepsilon_k})_{x_i}\right|^p\,dx=\sum_{i=1}^N \frac{1}{p_i}\,\int_B \left|U_{x_i}\right|^p\,dx,\qquad i=1,\dots,N.
\end{equation}
We next observe that the weak convergence of \(\{u_{\varepsilon_k})_{x_i}\}_{k\in \mathbb{N}}\) to \(U_{x_i}\) and the lower semicontinuity of the \(L^{p_i}\) norm on \(L^{p_1}\) imply that
\begin{equation}
\label{rehum}
\liminf_{k\to +\infty} \int_{B} \left| \frac{(u_{\varepsilon_k})_{x_i}+U_{x_i}}{2}\right|^{p_i}\,dx \geq \int_{B} \left|U_{x_i}\right|^{p_i}\,dx,\qquad i=1,\dots,N.
\end{equation}
Moreover, by Clarkson's inequality for $p_i\ge 2$, one has
\[
\left\|\frac{(u_{\varepsilon_k})_{x_i}+U_{x_i}}{2}\right\|^{p_i}_{L^{p_i}(B)}+\left\|\frac{(u_{\varepsilon_k})_{x_i}-U_{x_i}}{2}\right\|^{p_i}_{L^{p_i}(B)}\le \frac{1}{2}\,\left(\|(u_{\varepsilon_k})_{x_i}\|^{p_i}_{L^{p_i}(B)}+\|U_{x_i}\|^{p_i}_{L^{p_i}(B)}\right).
\]
We divide by $p_i$, sum over $i=1,\dots,N$ and rely on \eqref{norms} and \eqref{rehum}  to obtain
\[
\lim_{k\to +\infty} \sum_{i=1}^N \frac{1}{p_i}\left\|(u_{\varepsilon_k})_{x_i}-U_{x_i}\right\|^{p_i}_{L^{p_i}(B)}=0.
\]
By using this into \eqref{poincare} with $i=1$ and using the strong convergence of $U_{\varepsilon_k}$ to $U$, we get
\[
\lim_{k\to +\infty}\left[ \left\|u_{\varepsilon_k}-U\right\|_{L^{p_1}(B)} + \sum_{i=1}^N \left\|(u_{\varepsilon_k})_{x_i}-U_{x_i}\right\|_{L^{p_i}(B)}\right]=0.
\]
Finally, we observe that we can repeat this argument with any subsequence of the original family $\{(u_{\varepsilon})_{\varepsilon>0}\}$. Thus the above limit holds true for the whole family $\{u_{\varepsilon}\}_{0<\varepsilon\le \varepsilon_0}$ instead of $\{u_{\varepsilon_k}\}_{k\in \mathbb{N}}$ and \eqref{troppoforte!} follows.
\end{proof}
We recall that our main result Theorem \ref{teo:lipschitz} is valid for {\it bounded} local minimizers. Thus, the following simple uniform $L^\infty$ estimate will be crucial.
\begin{prop}[Uniform $L^\infty$ estimate]
\label{prop:Linfty}
With the notation above, let us further assume that $U\in L^\infty(2\,B)$. Then for every $0<\varepsilon\le \varepsilon_0$ we have
\[
\|u_\varepsilon\|_{L^\infty(B)}\le \|U\|_{L^\infty(2\,B)}.
\]
\end{prop}
\begin{proof}
By the maximum principle, see for example \cite[Theorem 2.1]{St}, we have
\[
\max_{B} |u_{\varepsilon}| = \max_{\partial B} |u_{\varepsilon}| = \max_{\partial B} |U_{\varepsilon}|\le \max_{\overline B} |U_\varepsilon|.
\]
By recalling the construction of $U_\varepsilon$ and using the hypothesis on $U$, we get the desired conclusion.
\end{proof}
As in \cite{BBLV}, the following standard technical result will be useful. The proof can be found in \cite[Lemma 6.1]{Gi}, for example.
\begin{lm}
\label{lm:giusti}
Let $0<r<R$ and let $Z(t):[r,R]\to [0,\infty)$ be a bounded function. Assume that for $r\le t<s\le R$ we have
\[
Z(t)\le \frac{\mathcal{A}}{(s-t)^{\alpha_0}}+\frac{\mathcal{B}}{(s-t)^{\beta_0}}+\mathcal{C}+\vartheta\,Z(s),
\]
with $\mathcal{A},\mathcal{B},\mathcal{C}\ge 0$, $\alpha_0\ge \beta_0>0$ and $0\le \vartheta<1$. Then we have
\[
Z(r)\le \left(\frac{1}{(1-\lambda)^{\alpha_0}}\,\frac{\lambda^{\alpha_0}}{\lambda^{\alpha_0}-\vartheta}\right)\,\left[\frac{\mathcal{A}}{(R-r)^{\alpha_0}}+\frac{\mathcal{B}}{(R-r)^{\beta_0}}+\mathcal{C}\right],
\]
where $\lambda$ is any number such that
\[
\vartheta^\frac{1}{\alpha_0}<\lambda<1.
\]
\end{lm}

\section{Caccioppoli-type inequalities}

\label{sec:caccioppoli}

The solution $u_\varepsilon$ of the problem \eqref{approximated} satisfies the Euler-Lagrange equation
\begin{equation}
\label{regolareg}
\sum_{i=1}^N \int g'_{i,\varepsilon}((u_\varepsilon)_{x_i})\, \varphi_{x_i}\, dx=0,\qquad \mbox{ for every }\varphi\in W^{1,\mathbf{p}}_0(B).
\end{equation}
From now on we will systematically suppress the subscript $\varepsilon$ on $u_\varepsilon$ and {\it simply write $u$}.
\par
In order to prove the gradient regularity, we need the equation satisfied by the gradient $\nabla u$.
Thus we insert a test function of the form $\varphi=\psi_{x_j}\in W^{1,\mathbf{p}}_0(B)$ in \eqref{regolareg}, compactly supported in $B$. After an integration by parts, we get
\begin{equation}
\label{derivatag}
\sum_{i=1}^N \int g_{i,\varepsilon}''(u_{x_i})\, u_{x_i\,x_j}\, \psi_{x_i}\, dx=0,
\end{equation}
for $j=1,\dots,N$. We thus found the equation solved by $u_{x_j}$. Observe that we are legitimate to integrate by parts, since $u\in C^2(B)$ by Lemma \ref{lm:below}.
\vskip.2cm\noindent
The following Caccioppoli inequality can be proved exactly as \cite[Lemma 3.2]{BBJ}, we omit the details.
\begin{lm}
\label{lm:standardcaccio}
Let $\Phi:\mathbb{R}\to\mathbb{R}^+$ be a $C^1$ convex function. 
Then there exists a constant $C=C(\mathbf{p})>0$ such that for every function $\eta\in C^\infty_0(B)$ and every $j=1,\dots,N$, we have
\begin{equation}
\label{mothergsob}
\sum_{i=1}^N \int g''_{i,\varepsilon}(u_{x_i})\,\left|\left(\Phi(u_{x_j})\right)_{x_i}\right|^2\, \eta^2\, dx \le C\,\sum_{i=1}^N \int g''_{i,\varepsilon}(u_{x_i})\,|\Phi(u_{x_j})|^2\, \eta_{x_i}^2\, dx.
\end{equation}
\end{lm}
Actually, we can drop the requirement that $\Phi$ has to be convex, under some circumstances. The resulting Caccioppoli inequality is of interest.
\begin{lm}
\label{lm:negativepower}
Let $-1<\alpha\le 0$. For every function $\eta\in C^\infty_0(B)$ and every $j=1,\dots,N$, we have
\[
\sum_{i=1}^N \int g_{i,\varepsilon}''(u_{x_i})\, u_{x_i\,x_j}^2\,|u_{x_j}|^\alpha\,\eta^2\,dx\le \frac{4}{(1+\alpha)^2}\,\sum_{i=1}^N\int g_{i,\varepsilon}''(u_{x_i})\, |u_{x_j}|^{\alpha+2}\,|\eta_{x_i}|^2\,dx.
\]
\end{lm}
When \(\alpha<0\), in the left hand side of the above inequality,  the quantity \(u_{x_i\,x_j}^2\,|u_{x_j}|^\alpha\) is defined to be \(0 \) on the set where \(u_{x_j}\) vanishes.
\begin{proof}
Let $\kappa>0$. We take in \eqref{derivatag} the test function
\[
\psi=u_{x_j}\,(\kappa+|u_{x_j}|^2)^\frac{\alpha}{2}\,\eta^2,
\]
where $\eta$ is as in the statement. We get
\[
\begin{split}
\sum_{i=1}^N \int& g_{i,\varepsilon}''(u_{x_i})\, u_{x_i\,x_j}^2\,(\kappa+|u_{x_j}|^2)^\frac{\alpha}{2}\,\eta^2\,dx\\
&+\alpha\,\sum_{i=1}^N \int g_{i,\varepsilon}''(u_{x_i})\, u_{x_i\,x_j}^2\,(\kappa+|u_{x_j}|^2)^\frac{\alpha-2}{2}\,|u_{x_j}|^2\,\eta^2\,dx\\
&=-2\,\sum_{i=1}^N\int g_{i,\varepsilon}''(u_{x_i})\, u_{x_i\,x_j}\,(\kappa+|u_{x_j}|^2)^\frac{\alpha}{2}\,u_{x_j}\,\eta\,\eta_{x_i}\,dx.
\end{split}
\]
We observe that 
\[
(\kappa+|u_{x_j}|^2)^\frac{\alpha-2}{2}\,|u_{x_j}|^2\le (\kappa+|u_{x_j}|^2)^\frac{\alpha}{2}\qquad \mbox{ and }\qquad (\kappa+|u_{x_j}|^2)^\frac{\alpha}{2}\,|u_{x_j}| \le (\kappa+|u_{x_j}|^2)^\frac{\alpha+1}{2}.
\]
From the previous identity, we get (remember that $\alpha$ is non positive)
\[
(1+\alpha)\,\sum_{i=1}^N \int g_{i,\varepsilon}''(u_{x_i})\, u_{x_i\,x_j}^2\,(\kappa+|u_{x_j}|^2)^\frac{\alpha}{2}\,\eta^2\,dx\le 2\,\sum_{i=1}^N\int g_{i,\varepsilon}''(u_{x_i})\, |u_{x_i\,x_j}|\,(\kappa+|u_{x_j}|^2)^\frac{\alpha+1}{2}\,|\eta|\,|\eta_{x_i}|\,dx.
\]
By using Young's inequality, we can absorb the Hessian term in the right-hand side, to get
\[
\sum_{i=1}^N \int g_{i,\varepsilon}''(u_{x_i})\, u_{x_i\,x_j}^2\,(\kappa+|u_{x_j}|^2)^\frac{\alpha}{2}\,\eta^2\,dx\le \frac{4}{(1+\alpha)^2}\,\sum_{i=1}^N\int g_{i,\varepsilon}''(u_{x_i})\, (\kappa+|u_{x_j}|^2)^\frac{\alpha+2}{2}\,|\eta_{x_i}|^2\,dx.
\]
By taking the limit as $\kappa$ goes to $0$ on both sides, and using Fatou Lemma on the left-hand side and the Dominated Convergence Theorem in the right-hand side, we get the conclusion.
\end{proof}
As in the standard growth case $p_1=\dots=p_N=p$, a key role is played by the following sophisticated Caccioppoli-type inequality for the gradient. The proof is the same as that of \cite[Proposition 3.2]{BBLV} and we omit it. It is sufficient to observe that the proof in \cite{BBLV} does not depend on the particular form of the functions $g_{i,\varepsilon}$.
 \begin{prop}[Weird Caccioppoli inequality]
\label{prop:weird}
Let \(\Phi, \Psi :[0,+\infty)\to [0,+\infty)\) be two non-decreasing continuous functions. We further assume that \(\Psi\) is convex and $C^1$. 
Let \(\eta\in C^{\infty}_0(B)\) and \(0\le \theta\le 2\), then for every \(k,j=1, \dots, N\),
\begin{equation}
\label{chiaraf}
\begin{split}
\sum_{i=1}^N \int g_{i,\varepsilon}''(u_{x_i})&\,u_{x_i x_j}^2\,\Phi(u_{x_j}^2)\,\Psi(u_{x_k}^2)\,\,\eta^2\,dx\\
&\le C\,\sum_{i=1}^N \int g_{i,\varepsilon}''(u_{x_i})\,u_{x_j}^2\,\Phi(u_{x_j}^2)\,\Psi(u_{x_k}^2)\,|\nabla \eta|^2\,dx\\
&+C\,\left(\sum_{i=1}^N \int g_{i,\varepsilon}''(u_{x_i})\,u_{x_i x_j}^2\,u_{x_j}^2\,\Phi(u_{x_j}^2)^2\,\Psi'(u_{x_k}^2)^\theta \,\eta^2\,dx\right)^\frac{1}{2}\\
&\times \left(\sum_{i=1}^N\int g_{i,\varepsilon}''(u_{x_i})\,|u_{x_k}|^{2\theta}\,\Psi(u_{x_k}^2)^{2-\theta}\,|\nabla \eta|^2\,dx \right)^\frac{1}{2}.
\end{split}
\end{equation}
\end{prop}

\section{Local energy estimates for the regularized problem}
\label{sec:leerp}

\subsection{Towards an iterative Moser's scheme}
We recall that
\begin{equation}
\label{gisecond}
g_{i,\varepsilon}''(t)=(p_i-1)\,|t|^{p_i-2}+\varepsilon.
\end{equation}
We use Proposition \ref{prop:weird} with the following choices
\begin{equation}
\label{sceltechiare}
\Phi(t)=t^{s-1}\qquad \mbox{ and }\qquad \Psi(t)=t^m,\qquad \mbox{ for }t\ge 0,
\end{equation}
with $1\le s \le m$. The proof of the following result is exactly the same as that of \cite[Proposition 4.1]{BBLV}.
\begin{prop}[Staircase to the full Caccioppoli]\label{stair}
Let $2\le p_1\le p_2\le \dots\le p_N$ and let \(\eta\in C^{\infty}_0(B)\). Then for every \(k,j=1, \dots, N\) and $1\le s\le m$,
\begin{equation}
\label{powerchiara}
\begin{split}
\sum_{i=1}^N \int_{\Omega} g_{i,\varepsilon}''(u_{x_i})\,u_{x_i x_j}^2\,|u_{x_j}|^{2\,s-2}\,|u_{x_k}|^{2\,m}\,\,\eta^2\,dx&\le C\,\sum_{i=1}^N \int g_{i,\varepsilon}''(u_{x_i})\,|u_{x_j}|^{2\,s+2\,m}\,|\nabla \eta|^2\,dx\\
& + C\,(m+1)\,\sum_{i=1}^N \int g_{i,\varepsilon}''(u_{x_i})\,|u_{x_k}|^{2\,s+2\,m}\,|\nabla \eta|^2\,dx\\
& + \sum_{i=1}^N \int g_{i,\varepsilon}''(u_{x_i})\,u_{x_i x_j}^2\,|u_{x_j}|^{4\,s-2}\,|u_{x_k}|^{2\,m-2\,s}\,\eta^2\,dx.
\end{split}
\end{equation}
\end{prop}
By iterating a finite number of times the previous estimate, we get the following
\begin{prop}[Caccioppoli for power functions]
\label{prop-russian}
Take an exponent $q$ of the form
\[
q=2^{\ell_0}-1,\qquad \mbox{ for a given } \ell_0\in\mathbb{N}\setminus\{0\}.
\] 
Let $2\le p_1\le p_2\le \dots\le p_N$  and let \(\eta\in C^{\infty}_0(B)\). Then for every \(k=1, \dots, N\), we have
\begin{equation}
\label{russiancircles}
\begin{split}
\int \left|\nabla \left(|u_{x_k}|^{q+\frac{p_k-2}{2}}\,u_{x_k}\right)\right|^2\,\eta^2\,dx
&\le C\,q^5\,\sum_{i,j=1}^N\int g_{i,\varepsilon}''(u_{x_i})\,|u_{x_j}|^{2\,q+2}\,|\nabla \eta|^2\,dx\\
& + C\,q^5\,\sum_{i=1}^N\int g_{i,\varepsilon}''(u_{x_i})\,|u_{x_k}|^{2\,q+2}\,|\nabla \eta|^2\,dx,
\end{split}
\end{equation}
for some $C=C(N,p_k)>0$.
\end{prop}
\begin{proof}
The proof is essentially the same as that of \cite[Proposition 4.2]{BBLV}.
We define the two finite families of indices $\{s_\ell\}$ and $\{m_\ell\}$ through
\[
s_\ell=2^\ell,\qquad m_{\ell}=q+1-2^{\ell},\qquad \ell\in\{0,\dots,\ell_0\}.
\]
By definition, we have
\[
1\le s_\ell\le m_\ell,\qquad \ell\in\{0,\dots,\ell_0-1\},
\]
\[
s_\ell+m_\ell=q+1,\qquad \ell\in\{0,\dots,\ell_0\},
\]
\[
4\,s_\ell-2=2\,s_{\ell+1}-2,\qquad 2\,m_\ell-2\,s_\ell=2\,m_{\ell+1},
\]
and
\[
s_0=1,\qquad m_0=q,\qquad s_{\ell_0}=2^{\ell_0},\qquad m_{\ell_0}=0. 
\]
From inequality \eqref{powerchiara}, we get for every $\ell\in\{0,\dots,\ell_0-1\}$,
\[
\begin{split}
\sum_{i=1}^N \int &g_{i,\varepsilon}''(u_{x_i})\,u_{x_i x_j}^2\,|u_{x_j}|^{2\,s_\ell-2}\,|u_{x_k}|^{2\,m_\ell}\,\,\eta^2\,dx\\
&\le C\,\sum_{i=1}^N \int g_{i,\varepsilon}''(u_{x_i})\,|u_{x_j}|^{2\,q+2}\,|\nabla \eta|^2\,dx\\
&+ C\,(m_\ell+1)\,\sum_{i=1}^N \int g_{i,\varepsilon}''(u_{x_i})\,| u_{x_k}|^{2\,q+2}\,|\nabla \eta|^2\,dx\\
&+\sum_{i=1}^N \int g_{i,\varepsilon}''(u_{x_i})\,u_{x_i x_j}^2\,|u_{x_j}|^{2\,s_{\ell+1}-2}\,|u_{x_k}|^{2\,m_{\ell+1}}\,\eta^2\,dx,
\end{split}
\]
for some $C>0$ universal.
By starting from $\ell=0$ and iterating the previous estimate up to $\ell=\ell_0-1$, we then get
\[
\begin{split}
\sum_{i=1}^N \int g_{i,\varepsilon}''(u_{x_i})\,u_{x_i x_j}^2\,|u_{x_k}|^{2\,q}\,\eta^2\,dx&\le C\,q^2\,\sum_{i=1}^N \int g_{i,\varepsilon}''(u_{x_i})\,|u_{x_j}|^{2\,q+2}\,|\nabla \eta|^2\,dx\\
&+ C\,q^2\,\sum_{i=1}^N\int g_{i,\varepsilon}''(u_{x_i})\,|u_{x_k}|^{2\,q+2}\,|\nabla \eta|^2\,dx\\
&+\sum_{i=1}^N \int g_{i,\varepsilon}''(u_{x_i})\,u_{x_i x_j}^2\,|u_{x_j}|^{2\,q}\,\eta^2\,dx,
\end{split}
\]
for a universal constant $C>0$.
For the last term, we apply the Caccioppoli inequality \eqref{mothergsob}  with
\[
\Phi(t)=\frac{|t|^{q+1}}{q+1},\qquad t\in\mathbb{R},
\]
thus we get
\[
\begin{split}
\sum_{i=1}^N \int g_{i,\varepsilon}''(u_{x_i})\,u_{x_i x_j}^2\,|u_{x_k}|^{2\,q}\,\eta^2\,dx
&\le C\,q^2\,\sum_{i=1}^N\int g_{i,\varepsilon}''(u_{x_i})\,|u_{x_j}|^{2\,q+2}\,|\nabla \eta|^2\,dx\\
& + C\,q^2\,\sum_{i=1}^N\int g_{i,\varepsilon}''(u_{x_i})\,|u_{x_k}|^{2\,q+2}\,|\nabla \eta|^2\,dx\\
&+\frac{C}{(q+1)^2}\,\sum_{i=1}^N\int g_{i,\varepsilon}''(u_{x_i})\,|u_{x_j}|^{2\,q+2}\,|\nabla \eta|^2\,dx;
\end{split}
\]
that is,
\begin{equation}
\label{eq_yoyo}
\begin{split}
\sum_{i=1}^N \int g_{i,\varepsilon}''(u_{x_i})\,u_{x_i x_j}^2\,|u_{x_k}|^{2\,q}\,\eta^2\,dx
&\le C\,q^2\,\sum_{i=1}^N\int g_{i,\varepsilon}''(u_{x_i})\,|u_{x_j}|^{2\,q+2}\,|\nabla \eta|^2\,dx\\
& + C\,q^2\,\sum_{i=1}^N\int g_{i,\varepsilon}''(u_{x_i})\,|u_{x_k}|^{2\,q+2}\,|\nabla \eta|^2\,dx,
\end{split}
\end{equation}
possibly for a different universal constant $C>0$.
\par
We now recall \eqref{gisecond}, thus we get
\[
\begin{split}
\sum_{i=1}^N \int g_{i,\varepsilon}''(u_{x_i})\,u_{x_i x_j}^2\,|u_{x_k}|^{2\,q}\,\eta^2\,dx&\ge \int |u_{x_k}|^{p_k-2}\,u_{x_k x_j}^2\,|u_{x_k}|^{2\,q}\,\eta^2\,dx\\
&=\left(\frac{2}{2\,q+p_k}\right)^2\,\int \left|\left(|u_{x_k}|^{q+\frac{p_k-2}{2}}\,u_{x_k}\right)_{x_j}\right|^2\,\eta^2\,dx.
\end{split}
\]
We can sum over $j=1,\dots,N$ to obtain
\[
\sum_{i,j=1}^N \int {g_{i,\varepsilon}''(u_{x_i})}\,u_{x_i x_j}^2\,|u_{x_k}|^{2\,q}\,\eta^2\,dx\ge \left(\frac{2}{2\,q+p_k}\right)^2\,\int \left|\nabla \left(|u_{x_k}|^{q+\frac{p_k-2}{2}}\,u_{x_k}\right)\right|^2\,\eta^2\,dx.
\]
This proves the desired inequality.
\end{proof}

\subsection{Towards higher integrability}

In order to prove the higher integrability of the gradient, we will need the following self-improving estimate. This is analogous to the estimate at the basis of \cite[Theorem 1.1]{BFZ}, which deals with the case $p_1=\dots=p_{N-1}<p_N$ only. As pointed out in the Introduction, our case will be much more involved.
\begin{prop}
\label{prop:BFZ}
For every $\alpha>-1$ and every $k=1,\dots,N$, there exists a constant $C=C(N,p_k,\alpha)>0$ such that for every pair of concentric balls $B_{r_0}\subset B_{R_0}\Subset B$, we have
\begin{equation}
\label{eq_start_algo}
\begin{split}
\int_{B_{r_0}}  | u_{x_k}|^{p_k+2+\alpha}\,dx &\leq C\,R_0^N\,\left(\left(\frac{\|u\|_{L^\infty(B)}}{R_0-r_0}\right)^{p_k+2+\alpha}+\varepsilon_0\right)\\
&+C\,\left(\frac{\|u\|_{L^\infty(B)}}{R_0-r_0}\right)^{\frac{2}{p_k}\,(p_k+2+\alpha)}\,\int_{B_{R_0}} \sum_{i\not= k} |u_{x_i}|^{\frac{p_i-2}{p_k}\,(p_k+2+\alpha)}\,dx.
\end{split}
\end{equation}
\end{prop}
\begin{proof}
We fix $k\in \{1, \dots, N\}$ and take \(\eta\in C^{\infty}_0(B)\) a positive cut-off function. For every \(\alpha >-1\), we  estimate  the quantity 
\[
\int | u_{x_k}|^{p_k+2+\alpha}\,\eta^2\,\,dx = \int u_{x_k} u_{x_k}|u_{x_k}|^{p_k+\alpha}\eta^2 \,dx.
\]
By integration by parts (recall that $u\in C^2(B)$), one gets
\[
\begin{split}
\int |u_{x_k}|^{p_k+2+\alpha}\,\eta^2\,dx &= - \int u\,\Big(u_{x_k}|u_{x_k}|^{p_k+\alpha}\,\eta^2\Big)_{x_k}\,dx \\
&=- (p_k+\alpha+1)\int u\, u_{x_k x_k}\,|u_{x_k}|^{p_k+\alpha}\eta^2\,dx - 2\,\int u\, u_{x_k}\,|u_{x_k}|^{p_k+\alpha}\eta\,\eta_{x_k}\,dx.
\end{split}
\]
Hence, we have
\begin{equation}\label{eq53}
\int | u_{x_k}|^{p_k+2+\alpha}\,\eta^2\,dx \leq (p_k+\alpha+1)\,\|u\|_{L^{\infty}(B)}\,\left(\int  |u_{x_k x_k}|\,|u_{x_k}|^{p_k+\alpha}\,\eta^2\,dx + \int |u_{x_k}|^{p_k+\alpha +1}\,\eta\, |\nabla \eta|\,dx\right).
\end{equation}
We now use the Young's inequality for the two terms in the right-hand side: for every \(\tau>0\), 
\[
|u_{x_k x_k}|\,|u_{x_k}|^{p_k+\alpha} \leq \tau\, |u_{x_k}|^{p_k+\alpha+2} + \frac{1}{4\,\tau}\, |u_{x_k}|^{p_k+\alpha-2}|\,u_{x_kx_k}|^2
\]
and 
\[
|u_{x_k}|^{p_k+\alpha +1}\,\eta \,|\nabla \eta|\leq \tau\, |u_{x_k}|^{p_k+\alpha+2}\, \eta^2 + \frac{1}{4\,\tau}\, |u_{x_k}|^{p_k+\alpha} |\nabla \eta|^2.
\]
In the first inequality, when \(p_k+\alpha-2<0\), the quantity \(|u_{x_k}|^{p_k+\alpha-2}|u_{x_kx_k}|^2\) is defined to be \(0\) on the set where \(u_{x_k}=0\).

Inserting these two inequalities into \eqref{eq53} and choosing 
\[
\tau=\frac{1}{4\,(p_k+\alpha+1)\,\|u\|_{L^{\infty}(B)}},
\] 
we can absorb the two terms multiplied by \(\tau\) in the left-hand side. This leads to
\[
\int \eta^2\, | u_{x_k}|^{p_k+2+\alpha}\,dx \leq C\, \|u\|_{L^{\infty}(B)}^2\,\left(\int  |u_{x_k x_k}|^2\,|u_{x_k}|^{p_k+\alpha-2}\,\eta^2 \,dx+ \int |u_{x_k}|^{p_k+\alpha}\, |\nabla \eta|^2\,dx\right),
\]
for a constant $C=C(p_k,\alpha)>0$.
Observe that 
\[
|u_{x_k x_k}|^2\,|u_{x_k}|^{p_k+\alpha-2}=|u_{x_kx_k}|^2\,|u_{x_k}|^{p_k-2}\,|u_{x_k}|^\alpha\le g_{k,\varepsilon}''(u_{x_k})\,|u_{x_kx_k}|^2\,|u_{x_k}|^\alpha.
\]
Thus, if $\alpha> 0$, we can apply the Caccioppoli inequality of Lemma \ref{lm:standardcaccio} with the convex function \(\Phi(t)=|t|^{\frac{\alpha}{2}+1}\). Otherwise, if $-1<\alpha\le 0$, we can apply Lemma \ref{lm:negativepower}. This gives
\[
\begin{split}
\int |u_{x_k x_k}|^2\,|u_{x_k}|^{p_k+\alpha-2}\,\eta^2 \,dx &\leq C\, \sum_{i=1}^{N}\int_{B} g_{i,\varepsilon}''(u_{x_i})\,|u_{x_k}|^{\alpha+2}\,|\eta_{x_i}|^2\,dx,
\end{split}
\]
and thus we obtain
\[
\begin{split}
\int \eta^2 \,| u_{x_k}|^{p_k+2+\alpha}\,dx &\leq C\,\|u\|_{L^{\infty}(B)}^2 \, \int  \left(|u_{x_k}|^{\alpha+2}\,\sum_{i=1}^{N}g_{i,\varepsilon}''(u_{x_i})+|u_{x_k}|^{p_k+\alpha} \right)|\nabla \eta|^2\,dx\\
&\le C\,\|u\|_{L^{\infty}(B)}^2 \, \int  \left(|u_{x_k}|^{\alpha+2}\,\sum_{i\not =k}|u_{x_i}|^{p_i-2}+|u_{x_k}|^{p_k+\alpha} +\varepsilon\,|u_{x_k}|^{\alpha+2}\right)|\nabla \eta|^2\,dx,
\end{split}
\]
where in the second inequality the constant $C$ may differ from the previous one. There we used \eqref{gisecond}.
\par
We now fix a pair concentric balls $B_r\subset B_R\Subset B$. Applying the above estimate to a non negative cut-off function \(\eta\in C^{\infty}_0(B_R)\) such that 
\[
\eta\equiv 1 \mbox{ on } B_r\qquad \mbox{ and }\quad \|\nabla \eta\|_{L^{\infty}(B_R)}\leq \frac{C}{R-r},
\] 
one gets
\begin{equation}
\label{eq86}
\begin{split}
\int_{B_r}  | u_{x_k}|^{p_k+2+\alpha}\,dx &\leq   \frac{C\,\|u\|_{L^{\infty}(B)}^2}{(R-r)^2}\int_{B_R}  \left(|u_{x_k}|^{\alpha+2}\sum_{i\not= k}|u_{x_i}|^{p_i-2} +   |u_{x_k}|^{p_k+\alpha}+\varepsilon\,|u_{x_k}|^{\alpha+2}\right)\,dx.
\end{split}
\end{equation}
We now want to absorb all the terms containing $u_{x_k}$ from the right-hand side.
Thus, we apply again the Young's inequality. For every \(\tau>0\), there exists \(C_0>0\) which depends only on \(N,p_k\) and \(\alpha\) such that
\[
|u_{x_k}|^{\alpha+2}\sum_{i\not= k}|u_{x_i}|^{p_i-2}\leq \tau\, |u_{x_k}|^{p_k+\alpha+2}+ \frac{C_0}{\tau^{\frac{\alpha+2}{p_k}}}\, \sum_{i\not= k}|u_{x_i}|^{(p_i-2)\,\frac{p_k+\alpha+2}{p_k}},
\]
and
\[
 |u_{x_k}|^{p_k+\alpha} \leq \tau\, |u_{x_k}|^{p_k+\alpha+2} + \frac{C_0}{\tau^{\frac{p_k+\alpha}{2}}}.
\]
Moreover, we use that 
\[
\varepsilon\,|u_{x_k}|^{\alpha+2}\le \varepsilon+ |u_{x_k}|^{p_k+\alpha+2}.
\]
thanks to the fact that $\varepsilon<1$ and $p_k\ge 2$.
Inserting these inequalities into \eqref{eq86} and choosing 
\[
\tau=\frac{(R-r)^2}{4\,C\,\|u\|^2_{L^{\infty}(B_R)}},
\] 
one obtains 
\[
\begin{split}
\int_{B_r}  | u_{x_k}|^{p_k+2+\alpha}\,dx &\leq \frac{1}{2}\int_{B_R}  | u_{x_k}|^{p_k+2+\alpha}\,dx\\
&+\frac{C\,\|u\|_{L^{\infty}(B_R)}^2}{(R-r)^2}\left(\frac{R^N}{\tau^{\frac{p_k+\alpha}{2}}}+\frac{1}{\tau^{\frac{\alpha+2}{p_k}}}\,\sum_{i\not= k}\int_{B_R}  |u_{x_i}|^{(p_i-2)\frac{p_k+2+\alpha}{p_k}}\,dx+\varepsilon\,R^N\right).
\end{split}
\]
By recalling the choice of $\tau$ above, this is the same as
\[
\begin{split}
\int_{B_r}  | u_{x_k}|^{p_k+2+\alpha}\,dx &\leq \frac{1}{2}\int_{B_R}  | u_{x_k}|^{p_k+2+\alpha}\,dx+C\,R^N\,\left(\left(\frac{\|u\|_{L^\infty(B)}}{R-r}\right)^{p_k+\alpha+2}+\varepsilon_0\right)\\
&+C\,\left(\frac{\|u\|_{L^\infty(B)}}{R-r}\right)^{2\,\frac{\alpha+2+p_k}{p_k}}\,\left(\int_{B_R} \sum_{i\not= k} |u_{x_i}|^{(p_i-2)\frac{p_k+2+\alpha}{p_k}}\,dx\right).
\end{split}
\]
We now fix $r_0<R_0$ as in the statement and use the previous estimate for $r_0\le r<R\le R_0$.
By applying Lemma \ref{lm:giusti}, one finally obtains that
\[
\begin{split}
\int_{B_{r_0}}  | u_{x_k}|^{p_k+2+\alpha}\,dx &\leq C\,R_0^N\,\left(\left(\frac{\|u\|_{L^\infty(B)}}{R_0-r_0}\right)^{p_k+\alpha+2}+\varepsilon_0\right)\\
&+C\,\left(\frac{\|u\|_{L^\infty(B)}}{R_0-r_0}\right)^{2\,\frac{p_k+\alpha+2}{p_k}}\,\left(\int_{B_{R_0}} \sum_{i\not= k} |u_{x_i}|^{(p_i-2)\frac{p_k+2+\alpha}{p_k}}\,dx\right).
\end{split}
\]
Here, the constant \(C\) depends on \(N,p_k\) and  \(\alpha\). This concludes the proof.
\end{proof}

\section{A Lipschitz estimate}
\label{sec:5}
\begin{prop}
\label{prop:a_priori_estimate}
Let $2\le p_1\le \dots\le p_N$ and $0<\varepsilon\le \varepsilon_0$. There exist an exponent $\gamma\ge p_N+2$, two exponents $\Theta,\beta>1$ and a constant $C>0$ such that for every $B_{r_0}\subset B_{R_0}\Subset B$ with \(0<r_0<R_0\le 1\), 
\begin{equation}
\label{lipschitz}
\|\nabla u\|_{L^\infty(B_{r_0})}\leq \frac{C}{(R_0-r_0)^{\beta}}\, \left(\int_{B_{R_0}} |\nabla u|^{\gamma}\,dx+1\right)^{\frac{\Theta}{\gamma}}.
\end{equation}
The parameters $\gamma,\beta,\Theta$ and the constant $C$ are independent of $\varepsilon$.
\end{prop}
\begin{proof}
The proof is very similar to that of \cite[Theorem 5.1]{BBLV}, though some important technical modifications have to be taken into account.
For simplicity, we assume throughout the proof that $N\ge 3$, so in this case the Sobolev exponent $2^*$ is finite. Observe that the case $N=2$, which could be treated with minor modifications, is already contained in \cite[Theorem 1.4]{BLPV} (the proof there is different).
\par
As in \cite{BBLV}, we divide the proof into four steps. 
\vskip.2cm\noindent
{\bf Step 1: a first iterative scheme}. We can proceed as in \cite[Proposition 5.1, Step 1]{BBLV}
by replacing the term
\[
\int |\nabla \eta|^2\, |u_{x_k}|^{2\,q+p}\,dx,
\]
there, with the following one
\[
\int |\nabla \eta|^2\, |u_{x_k}|^{2\,q+p_k}\,dx.
\]
Then the relevant outcome is now
\begin{equation}
\label{pronti??}
\begin{split}
\left( \int \left(\sum_{k=1}^N |u_{x_k}|^{2\,q+p_k}\right)^{\frac{2^*}{2}}\,\eta^{2^*}\,dx\right)^{\frac{2}{2^*}}
&\leq C\,q^5 \sum_{i, k=1}^{N} \int g_{i,\varepsilon}''(u_{x_i})\,|u_{x_k}|^{2\,q+2}\, |\nabla \eta|^2\,dx\\
&+ C\, \int |\nabla \eta|^2\, \sum_{k=1}^{N}|u_{x_k}|^{2\,q+p_k} \,dx.
\end{split}
\end{equation}
We now introduce the function 
\[
\mathcal{U}(x):= \max_{k=1, \dots, N}|u_{x_k}(x)|,
\]
and observe that
\[
|u_{x_k}|^{2\,q+p_1}-1\le |u_{x_k}|^{2\,q+p_k}\le |u_{x_k}|^{2\,q+p_N}+1,\qquad \mbox{ for } k=1,\dots,N.
\]
This in turn gives
\[
\mathcal{U}^{2\,q+p_1}-1\leq \sum_{k=1}^{N}|u_{x_k}|^{2\,q+p_k} \leq N\, \mathcal{U}^{2\,q+p_N}+N.
\]
Also, we have that 
\[
g_{i,\varepsilon}''(u_{x_i})\, |u_{x_k}|^{2\,q+2}\leq C\,\mathcal{U}^{2\,q+p_N}+C,\qquad \mbox{ for every } 1\leq i, k \leq N.
\]
By further observing that
\[
\left(\mathcal{U}^{2\,q+p_1}\right)^\frac{2^*}{2}\le \left(1+\sum_{k=1}^{N}|u_{x_k}|^{2\,q+p_k}\right)^\frac{2^*}{2}\le C\,\left(1+\left(\sum_{k=1}^{N}|u_{x_k}|^{2\,q+p_k}\right)^\frac{2^*}{2}\right),
\]
we obtain from \eqref{pronti??}
\[
\left( \int \mathcal{U}^{\frac{2^*}{2}(2\,q+p_1)}\, \eta^{2^*}\right)^{\frac{2}{2^*}} \leq C\,q^5\, \int \mathcal{U}^{2\,q+p_N}|\nabla \eta|^2\,dx + C\,q^5\, \int |\nabla \eta|^2 \, dx+C\,q^5\,\left(\int \eta^{2^*}\,dx\right)^\frac{2}{2^*},
\]
for a possibly different $C=C(N,\mathbf{p})>1$. By using the Sobolev embedding $W^{1,2}_0(B)\hookrightarrow L^{2^*}(B)$
\[
\left(\int \eta^{2^*}\,dx\right)^\frac{2}{2^*}\le C\, \int |\nabla \eta|^2\,dx,
\]
thus the previous estimate leads to
\begin{equation}
\label{bo}
\left(\int \mathcal{U}^{\frac{2^*}{2}(2\,q+p_1)}\,\eta^{2^*}\,dx\right)^\frac{2}{2^*} 
\le C\, q^5 \,\int |\nabla \eta|^2\,\Big(\mathcal{U}^{2\,q+p_N}+1\Big)\,dx.
\end{equation}
We fix two concentric balls \(B_r\subset B_R \Subset B\), with $0<r<R\le 1$.  
Then for every pair of radius $r\le t<s\le R$ we take in \eqref{bo} a standard cut-off function
\begin{equation}
\label{eq_def_eta}
\eta\in C^\infty_0(B_s),\quad \eta\equiv 1\mbox{ on } B_t,\quad 0\le \eta\le 1,\quad \|\nabla \eta\|_{L^\infty}\le\frac{C}{s-t}.
\end{equation}
This yields
\begin{equation}
\label{pronti!}
\left(\int_{B_t} \mathcal{U}^{\frac{2^*}{2}(2\,q+p_1)}\,dx\right)^\frac{2}{2^*} 
\le C\, \frac{q^5}{(s-t)^2} \,\int_{B_s} \Big(\mathcal{U}^{2\,q+p_N}+1\Big)\,dx.
\end{equation}
We define the sequence of exponents
\[
\gamma_j=p_N+2^{j+2}-2,\qquad j\ge 0,
\]
and take in \eqref{pronti!} $q=2^{j+1}-1$. This gives for every $j\ge 0$,
\begin{equation}
\label{pronti!!}
\begin{split}
\left(\int_{B_{t}} \mathcal{U}^{\frac{2^*}{2}\,(\gamma_j+p_1-p_N)}\,dx\right)^\frac{2}{2^*}\le C\,\frac{2^{5\,j}}{(s-t)^2}\,\int_{B_{s}}\Big(\mathcal{U}^{\gamma_j}+1\Big)\,dx , 
\end{split}
\end{equation}
for a possibly different constant $C=C(N,\mathbf{p})>1$. Observe that we always have
\[
\gamma_j+p_1-p_N\ge 2^{j+2},\qquad j\in\mathbb{N},
\]
thanks to the definition of $\gamma_j$.
\vskip.2cm\noindent
{\bf Step 2: filling the gaps.} By using the definition of $\gamma_j$, it is not difficult to see that 
\[
\gamma_j<\frac{2^*}{2}\,(\gamma_j+p_1-p_N)\qquad \Longleftrightarrow\qquad j>\log_2\left(\frac{N-2}{2}\,(p_N-2)-\frac{N}{2}\,(p_1-2)\right)-2.
\]
Thus we introduce the starting index\footnote{We use the convention that 
\[
\log t=-\infty\qquad \mbox{ for } t\le 0.
\]
Observe that $j_0=0$ whenever
\[
\frac{N-2}{2}\,(p_N-2)-\frac{N}{2}\,(p_1-2)<4\qquad \mbox{ i.\,e. }\qquad p_N<2+\frac{N\,(p_1-2)+8}{N-2}.
\]}
\[
j_0=\min\left\{j\in\mathbb{N}\,:\, j>\log_2\left(\frac{N-2}{2}\,(p_N-2)-\frac{N}{2}\,(p_1-2)\right)-2\right\}.
\]
By definition, this entails that
\[
\gamma_{j-1}<\gamma_j<\frac{2^*}{2}\,(\gamma_j+p_1-p_N),\qquad \mbox{ for every } j\ge j_0+1.
\]
By interpolation in Lebesgue spaces, we obtain
\[
\int_{B_{t}} \mathcal{U}^{\gamma_j}\,dx\le \left(\int_{B_{t}} \mathcal{U}^{\gamma_{j-1}}\,dx\right)^\frac{\tau_j\,\gamma_j}{\gamma_{j-1}}\,\left(\int_{B_{t}} \mathcal{U}^{\frac{2^*}{2}\,(\gamma_{j}+p_1-p_N)}\,dx\right)^{\frac{2}{2^*}\,\frac{(1-\tau_j)\,\gamma_j}{\gamma_j+p_1-p_N}},
\]
where the interpolation exponent $0<\tau_j<1$ is given by
\[
\tau_j=\frac{\gamma_{j-1}}{\gamma_j}\,\frac{\dfrac{2^*}{2}\,(\gamma_j+p_1-p_N)-\gamma_j}{\dfrac{2^*}{2}\,(\gamma_j+p_1-p_N)-\gamma_{j-1}}.
\]
We now rely on \eqref{pronti!!} to get
\[
\begin{split}
\int_{B_{t}} \mathcal{U}^{\gamma_j}\,dx&\le \left(\int_{B_{t}} \mathcal{U}^{\gamma_{j-1}}\,dx\right)^\frac{\tau_j\,\gamma_j}{\gamma_{j-1}}\, \left(C\,\frac{2^{5\,j}}{(s-t)^2}\,\int_{B_{s}}\Big(\mathcal{U}^{\gamma_j}+1\Big)\,dx\right)^{\frac{(1-\tau_j)\,\gamma_j}{\gamma_j+p_1-p_N}}\\
&=\left[\left(C\,\frac{2^{5\,j}}{(s-t)^2}\right)^{\frac{1-\tau_j}{\tau_j}\,\frac{\gamma_j}{\gamma_j+p_1-p_N}}\,\left(\int_{B_{t}} \mathcal{U}^{\gamma_{j-1}}\,dx\right)^\frac{\gamma_j}{\gamma_{j-1}}\right]^{\tau_j}\, \left(\int_{B_{s}}\Big(\mathcal{U}^{\gamma_j}+1\Big)\,dx\right)^{\frac{(1-\tau_j)\,\gamma_j}{\gamma_j+p_1-p_N}}.
\end{split}
\]
By Young's inequality, for every $j\ge j_0+1$, we get
\begin{equation}
\label{suffering}
\begin{split}
\int_{B_{t}} \mathcal{U}^{\gamma_j}\,dx&\le \frac{(1-\tau_j)\,\gamma_j}{\gamma_j+p_1-p_N}\,\int_{B_{s}}\Big( \mathcal{U}^{\gamma_j}+1\Big)\,dx \\
&+ \frac{1}{\left(\dfrac{\gamma_j+p_1-p_N}{(1-\tau_j)\,\gamma_j}\right)'}\,\left(C\,\frac{2^{5\,j}}{(s-t)^2}\right)^{\frac{(1-\tau_j)\,\gamma_j}{\gamma_j+p_1-p_N}\,\left(\frac{\gamma_j+p_1-p_N}{(1-\tau_j)\,\gamma_j}\right)'}\,\left(\int_{B_{t}} \mathcal{U}^{\gamma_{j-1}}\,dx\right)^{\frac{\gamma_j}{\gamma_{j-1}}\,\tau_j\,\left(\frac{\gamma_j+p_1-p_N}{(1-\tau_j)\,\gamma_j}\right)'}.
\end{split}
\end{equation}
We also introduce the second index
\[
j_1=\min\left\{j\in\mathbb{N}\,:\, j>\log_2\left((N-2)\,(p_N-2)-N\,\left(p_1-2\right)\right)-2\right\}.
\]
If we finally set
\[
J=1+\max\{j_0,\,j_1\},
\]
then by Lemma \ref{lm:madonne} we know that
\begin{equation}
\label{mappazza}
0<C_1\le \frac{(1-\tau_j)\,\gamma_j}{\gamma_j+p_1-p_N}\le C_2<1,\qquad \mbox{ for every } j\ge J.
\end{equation}
This in turn implies that for every $j\ge J$
\[
\frac{1}{\left(\dfrac{\gamma_j+p_1-p_N}{(1-\tau_j)\,\gamma_j}\right)'}\le \frac{1}{\left(\dfrac{1}{C_1}\right)'}=1-C_1.
\]
Thus from \eqref{suffering}\, we get
\begin{equation}
\label{suffering2}
\begin{split}
\int_{B_{t}} \mathcal{U}^{\gamma_j}\,dx&\le C_2\,\int_{B_{s}}\Big( \mathcal{U}^{\gamma_j}+1\Big)\,dx \\
&+ (1-C_1)\,\left(C\,\frac{2^{5\,j}}{(s-t)^2}\right)^{\beta}\,\left(\int_{B_{t}} \mathcal{U}^{\gamma_{j-1}}\,dx\right)^{\frac{\gamma_j}{\gamma_{j-1}}\,\tau_j\,\left(\frac{\gamma_j+p_1-p_N}{(1-\tau_j)\,\gamma_j}\right)'}\\
\end{split}
\end{equation}
for some \(1<\beta<\infty\), depending on $N, p_1$ and $p_N$.
In the last inequality we also used that \(s\leq R\le 1\) and $C>1$, together with \eqref{mappazza}. Finally we set
\[
\varepsilon_j=\tau_j\,\left(\frac{\gamma_j+p_1-p_N}{(1-\tau_j)\,\gamma_j}\right)'-1,\qquad \mbox{ for } j\ge J,
\]
and rewrite \eqref{suffering2} as
\begin{equation}
\label{suffering3}
\begin{split}
\int_{B_{t}} \mathcal{U}^{\gamma_j}\,dx&\le C_2\,\int_{B_{s}}\mathcal{U}^{\gamma_j}\,dx \\
&+ (1-C_1)\,\left(C\,\frac{2^{5\,j}}{(s-t)^2}\right)^{\beta}\,\left(\int_{B_{t}} \mathcal{U}^{\gamma_{j-1}}\,dx\right)^{\frac{\gamma_j}{\gamma_{j-1}}\,(1+\varepsilon_j)}+C_2\,|B_R|,
\end{split}
\end{equation}
which holds for every $r\le s<t\le R$.
By applying Lemma \ref{lm:giusti} with
\[
Z(t)= \int_{B_{t}} \mathcal{U}^{\gamma_j}\,dx ,\qquad \alpha_0=2\, \beta, \qquad \mbox{ and }\qquad \vartheta=C_2,
\]
we finally obtain for every $j\ge J$,
\begin{equation}
\label{conj}
\int_{B_r} \mathcal{U}^{\gamma_j}\,dx\le C\,\left( 2^{5\,j\,\beta}\,(R-r)^{-2\,\beta}\,\left(\int_{B_R} \mathcal{U}^{\gamma_{j-1}}\,dx\right)^{\frac{\gamma_j}{\gamma_{j-1}}\,(1+\varepsilon_j)}+ 1\right),
\end{equation}
for some $C=C(N,p_1,p_N)>1$.
\vskip.2cm\noindent
{\bf Step 3: Moser's iteration.} We now iterate the previous estimate on a sequence of shrinking balls. We fix two radii $0<r<R\le 1$ and define the sequence 
\[
R_j=r+\frac{R-r}{2^{j-J}},\qquad j\ge J.
\]
We use \eqref{conj} with \(R_{j+1}<R_j\) in place of \(r<R\). 
Thus we get 
\begin{equation}
\label{scamone}
\int_{B_{R_{j+1}}} \mathcal{U}^{\gamma_j}\,dx
\le \,C\,\left(2^{7\,j\,\beta}\,(R-r)^{-2\,\beta}\left( \int_{B_{R_j}}\mathcal{U}^{\gamma_{j-1}}\,dx \right)^{\frac{\gamma_j}{\gamma_{j-1}}\,(1+\varepsilon_j)}+ 1\right)
\end{equation}
where the constant \(C>1\) depends on $N$ and $p_1,p_N$ only.
\par
We introduce the notation
\[
Y_j=\int_{B_{R_{j}}} \mathcal{U}^{\gamma_{j-1}}\,dx,
\]
thus \eqref{scamone} reads
\[
Y_{j+1} \le \,C\,\left(2^{7\,j\,\beta}\,(R-r)^{-2\,\beta}\,Y_{j}^{\frac{\gamma_j}{\gamma_{j-1}}\,(1+\varepsilon_j)}+ 1\right)
\le \left(C\,2^{7\,\beta}\,(R-r)^{-2\,\beta}\right)^j\,(Y_{j}+1)^{\frac{\gamma_j}{\gamma_{j-1}}\,(1+\varepsilon_j)}.
\]
Here, we have used again that $C>1$  and $R\le 1$, so that the term multiplying $Y_j$ is larger than $1$.
By iterating the previous estimate starting from $j=J$ and using some standard manipulations, we obtain
\[
\begin{split}
Y_{n+1}&\le \left(C\,2^{7\,\beta}\,(R-r)^{-2\,\beta}\right)^{n}\,(Y_n+1)^{\frac{\gamma_n}{\gamma_{n-1}}\,(1+\varepsilon_n)}\\
&\le \left(C\,2^{7\,\beta}\,(R-r)^{-2\,\beta}\right)^n\,\left(\left(C\,2^{7\,\beta}\,(R-r)^{-2\,\beta}\right)^{n-1}\,(Y_{n-1}+1)^{\frac{\gamma_{n-1}}{\gamma_{n-2}}\,(1+\varepsilon_{n-1})}+1\right)^{\frac{\gamma_n}{\gamma_{n-1}}\,(1+\varepsilon_n)}\\
&\le \left(C\,2^{7\,\beta}\,(R-r)^{-2\,\beta}\right)^n\,\left(2\,\left(C\,2^{7\,\beta}\,(R-r)^{-2\,\beta}\right)^{n-1}\,(Y_{n-1}+1)^{\frac{\gamma_{n-1}}{\gamma_{n-2}}\,(1+\varepsilon_{n-1})}\right)^{\frac{\gamma_n}{\gamma_{n-1}}\,(1+\varepsilon_n)}\\
&\le \dots\\
&\le \Big(2\,C\,2^{7\,\beta}\,(R-r)^{-2\,\beta}\Big)^{\sum\limits_{j=J}^{n}\left(j\,\frac{\gamma_n}{\gamma_{j}}\prod\limits_{k=j+1}^n(1+\varepsilon_k)\right)}\,\Big[Y_{J}+1\Big]^{\frac{\gamma_n}{\gamma_{J-1}}\,\prod\limits_{j=J}^n(1+\varepsilon_j)},
\end{split}
\]
where we used that $C\,2^{7\,\beta}\,(R-r)^{-2\,\beta}>1$. We now simply write $C$ in place of $2\,C\,2^{7\,\beta}$
and take the power $1/\gamma_n$ on both sides:
\[
\begin{split}
Y_{n+1}^\frac{1}{\gamma_n}&\le \Big(C\,(R-r)^{-2\,\beta}\Big)^{\sum\limits_{j=J}^{n}\frac{j}{\gamma_{j}}\prod\limits_{k=j+1}^n(1+\varepsilon_k)}
\,\Big[Y_J+1\Big]^\frac{\prod\limits_{j=J}^n(1+\varepsilon_j)}{\gamma_{J-1}}\\
&\leq \Big(C\,(R-r)^{-2\, \beta}\Big)^{\Theta\sum\limits_{j=J}^{n}\frac{j}{\gamma_{j}}}
\,\Big[Y_J+1\Big]^\frac{\Theta}{\gamma_{J-1}}.
\end{split}
\]
In the previous estimate, we set
\[
\Theta=\lim_{n\to\infty}\prod_{j=0}^n(1+\varepsilon_j),
\]
which is a finite number, thanks to Lemma \ref{lm:putain}.
We observe that \(\gamma_{j}\sim 2^{j+2} \) as \(j\) goes to \(\infty\). This implies the convergence of the series above and we thus get
\[
\|\mathcal{U}\|_{L^{\infty}(B_{r})} =  \lim_{n\to\infty}\left(\int_{B_{R_{n+1}}} \mathcal{U}^{\gamma_{n+1}}\,dx\right)^\frac{1}{\gamma_{n+1}} \leq C\,  (R-r)^{-\beta'}\,\left(\int_{B_{R}} \mathcal{U}^{\gamma_{J-1}}\,dx+1\right)^\frac{\Theta}{\gamma_{J-1}},
\]
for some $C=C(N,p_1, p_N)>{ 1}$ and $\beta'=\beta'(N,p_1, p_N)>0$. 
By recalling the definition of \(\mathcal{U}\), we finally obtain
\[
\|\nabla u\|_{L^{\infty}(B_{r})} \leq C\,(R-r)^{-\beta'}\, \left(\int_{B_{R}} |\nabla u|^{\gamma_{J-1}}\,dx+1\right)^{\frac{\Theta}{\gamma_{J-1}}}.
\]
This concludes the proof.
\end{proof}

\section{A recursive gain of integrability for the gradient} 
\label{sec:6}

The main outcome of estimate \eqref{lipschitz} is the following: assume that our local minimizer $U$ has a gradient with a sufficiently high integrability, then one would be able to conclude that $\nabla U$ has to be bounded. 
We notice that since the explicit determination of the exponent $\gamma$ in \eqref{lipschitz} is actually very intricate (unless some upper bounds  on $p_N/p_1$ are imposed), we essentially  need to prove that
\[
\nabla U\in L^q_{\rm loc}\qquad \mbox{ for every }q<\infty,
\]
in order to be on the safe side. Thus, in order to infer the desired local Lipschitz regularity on $U$, we are going to prove a higher integrability estimate on $\nabla u_\varepsilon$, which is uniform with respect to $0<\varepsilon\le\varepsilon_0$. This is the content of the result of this section. Remember that  we simplify the notation $u_\varepsilon$ and replace it by $u$.
\begin{prop}
\label{prop:pierre}
Let $2\le p_1\le \dots\le p_N<+\infty$ and $0<\varepsilon\le \varepsilon_0$. For every $2\le q_0<+\infty$ and every $B_{R_0}\Subset B$, there exists a constant $C>0$ such that
\[
\sum_{i=1}^N \int_{B_{R_0}} |u_{x_i}|^{p_i\,q_0}\,dx\le C.
\]
The constant $C$ depends on $N,\mathbf{p},q_0, R_0, \mathrm{dist}(B_{R_0},\partial B)$,
\[
\|u\|_{L^\infty(B)}\qquad \mbox{ and }\qquad \sum_{i=1}^N \int_B |u_{x_i}|^{p_i}\,dx.
\]
\end{prop}
\begin{proof}
We proceed to exploit the scheme of Proposition \ref{prop:BFZ}. In what follows, we use the convention that
\[
\frac{p}{p-2}=+\infty,
\]
whenever $p=2$. We can also assume without loss of generality that 
\[
p_N>2,
\]
otherwise $p_1=\dots=p_N=2$ and in this case the regularity theory for our problem is well-established ($U$ would be a  harmonic function in such a situation).
\par
We fix $q_0$ as in the statement and introduce the exponents 
\begin{equation}
\label{qi}
q_j=\min\left\{\frac{p_j}{p_j-2},q_0\right\}=\min \left\{\left(\frac{p_j}{2}\right)',q_0\right\},\qquad j=1,\dots,N.
\end{equation}
Since \(p_1\leq\dots \leq  p_N\), we get that \(q_1\geq q_2 \geq \dots \geq q_N\) and thus
\begin{equation}
\label{ordineqi}
\min_{i=1,\dots,k} q_i = q_{k},\qquad \mbox{ for every } k\in\{1,\dots,N\}.
\end{equation}
We now prove by {\it downward induction} on \(j=N, \dots, 1\) that the following fact holds:
for every $B_R\Subset B$, one has
\begin{equation}
\label{provalo!}
\begin{array}{c}
\mbox{ for every $j\in\{1,\dots,N\}$}\\
\mbox{and every $B_R\Subset B$,}
\end{array}
\qquad \sum_{i=j}^N\int_{B_R} |u_{x_i}|^{p_i\,q_{j-1}}\,dx\le C,\qquad \mbox{ with $C>0$ independent of }\varepsilon.
\end{equation}
In particular, for $j=1$, \eqref{provalo!} implies
\[
\sum_{i=1}^N\int_{B_R} |u_{x_i}|^{p_i\,q_{0}}\,dx\le C,
\] 
for a uniform constant $C>0$. The statement on the quality of the constant $C$ will be clear from the computations below.
\vskip.2cm\noindent
{\bf Initialization step}. We start from $j=N$. We observe that in this case the right-hand side of \eqref{eq_start_algo} (written for \(k=N\))
is uniformly bounded with respect to $\varepsilon>0$, provided that \(\alpha>-1\) is chosen in such a way  that for every \(i\in \{1,\dots,N-1\}\), one has
\[
(p_i-2)\,\frac{p_N+2+\alpha}{p_N}\leq p_i.
\]
If $p_i=2$, this is automatically satisfied. Otherwise, this is equivalent to
\[
p_N+2+\alpha \leq p_N\,\frac{p_i}{p_i-2}.
\]
We define \(\alpha\) by 
\[
p_N+2+\alpha = p_N\, q_{N-1}.
\]
By definition of \(q_{N-1}\),
\[
p_N+2+\alpha = p_N\, \min \left\{\frac{p_{N-1}}{p_{N-1}-2}, q_0 \right\} \leq p_N\, \min_{1\leq i \leq N-1} \frac{p_i}{p_i -2},
\]
as desired.
We need to check that \(\alpha>-1\), or equivalently 
\begin{equation}\label{eq1235}
p_N\, \min \left\{\frac{p_{N-1}}{p_{N-1}-2} , q_0 \right\} >p_N+1.
\end{equation}
Since \(q_0\geq 2\), one has \(p_N\, q_0 >p_N+1\). Moreover, using that \(p_{N-1}\leq p_N\), one gets \(2\,p_{N}>p_{N-1}-2\) which in turn is equivalent to 
\[
p_N \frac{p_{N-1}}{p_{N-1}-2} >p_N+1.
\] 
This proves \eqref{eq1235}. 
\par
Thus for every $B_R\Subset B_{R'}\Subset B$, it follows from \eqref{eq_start_algo} that
\[
\begin{split}
\int_{B_R}  | u_{x_N}|^{p_N\,q_{N-1}}\,dx &\leq C\,(R')^N\,\left(\left(\frac{\|u\|_{L^\infty(B)}}{R'-R}\right)^{p_N\,q_{N-1}}+\varepsilon_0\right)\\
&+C\,\left(\frac{\|u\|_{L^\infty(B)}}{R'-R}\right)^{2\,q_{N-1}}\,\int_{B_{R'}} \sum_{i=1}^{N-1} |u_{x_i}|^{\frac{p_i-2}{p_N}\,(p_N\,q_{N-1})}\,dx.
\end{split}
\]
Since \(q_i\geq q_{N-1}\) for $1\le i\le N-1$, one has 
\[
\frac{p_i}{q_{N-1}\,(p_i-2)}\geq 1.
\] 
Using H\"older's inequality for each term of the sum of the right-hand side, with the exponent \(p_i/(q_{N-1}(p_i-2))\) and its conjugate, one gets
\[
\begin{split}
\int_{B_R}  | u_{x_N}|^{p_N\,q_{N-1}}\,dx &\leq
 C\,(R')^N\,\left(\left(\frac{\|u\|_{L^\infty(B)}}{R'-R}\right)^{p_N\,q_{N-1}}+\varepsilon_0\right)\\
&+C\,\left(\frac{\|u\|_{L^\infty(B)}}{R'-R}\right)^{2\,q_{N-1}}\,\sum_{i=1}^{N-1}\left(\int_{B_{R'}} |u_{x_i}|^{p_i}\,dx\right)^\frac{(p_i-2)\,q_{N-1}}{p_i}.
\end{split}
\]
By using Lemma \ref{lm:below} and Proposition \ref{prop:Linfty} in order to control the two terms on the right-hand side, we get a uniform (in $\varepsilon$) control on the $L^{p_N\,q_{N-1}}(B_R)$ norm of $u_{x_N}$.
This finally establishes the initialization step \(j=N\), i.e. 
\[
\mbox{for every $B_R\Subset B,$ we have }
\qquad \int_{B_R} |u_{x_N}|^{p_N\,q_{N-1}}\,dx\le C,\qquad \mbox{ with $C>0$ independent of }\varepsilon.
\]
\vskip.2cm\noindent
{\bf Inductive step}. We then assume that the assertion \eqref{provalo!} is true for some \(j\in \{2, \dots, N\}\) and we prove it for \(j-1\). By the induction assumption, we thus know that 
\begin{equation}
\label{eq146}
\begin{array}{c}
\mbox{ for some $j\in\{2,\dots,N\}$}\\
\mbox{and every $B_R\Subset B$,  }
\end{array}
\qquad \sum_{i=j}^N\int_{B_R} |u_{x_i}|^{p_i\,q_{j-1}}\,dx\le C,\qquad \mbox{ with $C>0$ independent of }\varepsilon.
\end{equation}
In the rest of the proof, we establish that \eqref{eq146} implies
\begin{equation}
\label{eq-aim}
\begin{array}{c}
\mbox{for every $B_R\Subset B$, }
\end{array}
\qquad \sum_{i=j-1}^N\int_{B_R} |u_{x_i}|^{p_i\,q_{j-2}}\,dx\le C,\qquad \mbox{ with $C>0$ independent of }\varepsilon.
\end{equation}
In order to prove this, as explained in the Introduction, we need to employ a {\it multiply iterative scheme} based on Proposition \ref{prop:BFZ}.
More specifically, we start by relying on \eqref{eq_start_algo} with the choices
\[
k=j-1\qquad \mbox{ and }\qquad p_{j-1}+2+\alpha= p_{j-1}\, \min \left\{\min_{0\le i\le j-2}q_i ,\, q_{j-1}\,\min_{j\le i\le N}q_i\right\}.
\]
We first justify the fact that such a choice for \(\alpha\) is feasible. Observe that
\[
\min_{0\le i\le j-2} q_{i}=q_{j-2} \qquad \mbox{ and } \qquad  \min_{j\le i\le N}q_i = q_N,
\]
hence the condition on $\alpha$ is equivalent to
\begin{equation}\label{eq149+}
p_{j-1}+2+\alpha = p_{j-1} \min \left\{q_{j-2},\, q_{j-1}\,q_N \right\}.
\end{equation}
Since 
\[
q_{j-1}=\min \left\{\frac{p_{j-1}}{p_{j-1}-2},q_0\right\} \mbox{ with } q_0\geq 2,
\] 
one has \(p_{j-1}+2 \le p_{j-1}\,q_{j-1}\). By recalling that the exponents $q_j$ are non increasing and larger than $1$, this implies that
\[
p_{j-1}+2\le p_{j-1}\, q_{j-2} \qquad \mbox{ and } \qquad p_{j-1}+2 <p_{j-1}\,q_{j-1}\, q_N,
\]
and thus 
\[
p_{j-1}+2+\alpha =p_{j-1}\,\min \{q_{j-2},\,q_{j-1}\,q_N\}\ge p_{j-1}+2.
\]
This implies that \(\alpha \geq 0\) as desired. 

We next rely on the fact that by Lemma \ref{lm:below} and Proposition \ref{prop:Linfty}, we have
\[
\|u\|_{L^\infty(B)}+\sum_{i=1}^{N}\int_B |u_{x_i}|^{p_i}\,dx\le C,
\] 
with a constant $C>0$ independent of $\varepsilon>0$ 
and on the induction assumption \eqref{eq146}, which gives a local uniform (in $\varepsilon$) control on
\[
\sum_{i=j}^N\int_{B_R} |u_{x_i}|^{p_i\,q_{j-1}}\,dx,
\]
for $B_R\Subset B$. Hence, the definition \eqref{eq149+} of \(\alpha\) ensures that the right-hand side of \eqref{eq_start_algo} is uniformly bounded. Thus from Proposition \ref{prop:BFZ}, we get that for every $B_R\Subset B$ we have
\begin{equation}
\label{eq178}
\int_{B_R} |u_{x_{j-1}}|^{\beta_{j-1}^{(0)}}\,dx\le C,
\end{equation}
with $C>0$ independent of $\varepsilon$. Here, the exponent $\beta_{j-1}^{(0)}$ is given by 
\[
\beta_{j-1}^{(0)}= p_{j-1}+2+\alpha =p_{j-1} \min \left\{q_{j-2},\, q_{j-1}q_N \right\}.
\]
We can summarize the previous integrability information as the following estimate: for every $B_R\Subset B,$
\begin{equation}\label{eq184}
\sum_{i=j-1}^N \int_{B_R} |u_{x_i}|^{\beta_{i}^{(0)}}\,dx\le C,
\end{equation}
with $C>0$ independent of $\varepsilon>0$ and
\begin{equation}
\label{boh}
\left\{\begin{array}{ccl}
\beta_{j-1}^{(0)}&=&p_{j-1} \min \left\{q_{j-2},\, q_{j-1}\,q_N \right\},\\
&&\\
\beta_{i}^{(0)}&=& p_i\, q_{j-1},\qquad \mbox{ for } i=j,\dots,N.\\
\end{array}
\right.
\end{equation}
We proceed to define by induction a vector sequence 
\[
\left(\beta_{j-1}^{(\ell)},\dots,\beta_N^{(\ell)}\right),\qquad \ell\in\mathbb{N},
\] 
as follows: for $\ell=0$ this is given by \eqref{boh} and then we use the following {\it multiply recursive} scheme
\begin{equation}\label{eq-def-beta}
\left\{\begin{array}{ccl}
\beta_N^{(\ell+1)}&=&p_N\,\displaystyle\min\left\{q_{j-2},\, \min_{j-1\le k\le N-1}\frac{\beta^{(\ell)}_k}{p_k-2}\right\}\\
\beta_{N-1}^{(\ell+1)}&=&p_{N-1}\,\displaystyle\min\left\{q_{j-2},\, \min_{j-1\le k\le N-2}\frac{\beta^{(\ell)}_k}{p_k-2},\, \frac{\beta^{(\ell+1)}_N}{p_N-2}\right\}\\
\beta_{N-2}^{(\ell+1)}&=&p_{N-2}\,\displaystyle\min\left\{q_{j-2},\, \min_{j-1\le k\le N-3}\frac{\beta^{(\ell)}_k}{p_k-2},\, \min_{N-1\le k\le N}\frac{\beta^{(\ell+1)}_k}{p_k-2}\right\}\\
\vdots & = &\vdots\\
\beta_{j-1}^{(\ell+1)}&=&p_{j-1}\,\displaystyle\min\left\{q_{j-2},\, \min_{j\le k\le N}\frac{\beta^{(\ell+1)}_k}{p_k-2}\right\}.
\end{array}
\right.
\end{equation}
We first observe that this scheme is well-defined, since each $\beta^{(\ell+1)}_i$ is determined either by $(\beta_{j-1}^{(\ell)},\dots,\beta^{(\ell)}_N)$ or by an updated information on the  $\beta^{(\ell+1)}_k$, with $k\ge i+1$.
Moreover, thanks to Lemma \ref{lm-beta-increasing} and Lemma \ref{lm:limitbetai} below, we have that
\[
\left\{\beta_{i}^{(\ell)}\right\}_{\ell\in \mathbb{N}}\qquad \mbox{ is nondecreasing, for every } j-1\le i\le N,
\]
and there exists $\ell_0\in\mathbb{N}$ such that 
\begin{equation}
\label{argh}
\mbox{ for every } \ell\ge \ell_0,\qquad \beta_i^{(\ell)}=p_i\,q_{j-2},\qquad \mbox{ for } i=j-1,\dots,N.
\end{equation}
With these definitions at hand, we now prove that 
\begin{equation}
\label{last}
\mbox{ for every } B_R\Subset B,\quad \sum_{i=j-1}^N\int_{B_R} |u_{x_i}|^{\beta_i^{(\ell)}}\,dx\le C,\ \mbox{ for every }\ell\in\mathbb{N},\qquad \mbox{ with $C>0$ independent of }\varepsilon.
\end{equation}
By taking into account \eqref{argh}, this will eventually establish \eqref{eq-aim}, thus concluding the proof.
\vskip.2cm\noindent
In turn, the proof of \eqref{last}  relies on an  induction argument. The assertion \eqref{last} is true for \(\ell=0\), thanks to \eqref{eq184}. 
\par
We now assume \eqref{last} to hold for some $\ell\in\mathbb{N}$ and establish the same for $\ell+1$, i.e. 
\[
\mbox{ for every }B_R\Subset B,  \qquad \sum_{i=j-1}^N \int_{B_R} |u_{x_i}|^{\beta_{i}^{(\ell+1)}}\,dx\le C,
\]
with $C$ independent of $\varepsilon$. Actually, by a downward induction on \(m=N, \dots, j-1\), we prove that 
\[
\sum_{i=m}^N \int_{B_R} |u_{x_i}|^{\beta_{i}^{(\ell+1)}}\,dx\le C,\qquad \mbox{ with $C$ independent of $\varepsilon$}.
\]
For the initialization step \(m=N\), we apply \eqref{eq_start_algo} with \(k=N\) and for  the following choice of \(\alpha\):
\[
p_N+2+\alpha=p_N\,\min\left\{q_{j-2},\,\min_{j-1\le k\le N-1}\frac{\beta^{(\ell)}_k}{p_k-2}\right\}=\beta^{(\ell+1)}_N.
\]
In order to justify  that  \(\alpha\) defined as such is non negative, we rely on the fact that for every \(i\in \{j-1, \dots, N\}\), the sequence \(\{\beta_{i}^{(\ell)}\}_{\ell\in \mathbb{N}}\) is nondecreasing. This implies that 
\[
\alpha \geq \beta_{N}^{(0)}-(p_N+2) = p_N\, q_{j-1} -(p_N+2) \geq p_N\, q_{N} -(p_N+2)\geq 0.
\]
Hence,  such a choice of $\alpha$ is feasible.
\par
We get that for every $B_R\Subset B_{R'}\Subset B$, 
\[
\begin{split}
\int_{B_{R}}  | u_{x_N}|^{\beta^{(\ell+1)}_N}\,dx &\leq C\,(R')^N\,\left(\left(\frac{\|u\|_{L^\infty(B)}}{R'-R}\right)^{\beta^{(\ell+1)}_N}+\varepsilon_0\right)\\
&+C\,\left(\frac{\|u\|_{L^\infty(B)}}{R'-R}\right)^{\frac{2}{p_N}\,\beta^{(\ell+1)}_N}\,\int_{B_{R'}} \sum_{i=1}^{j-2} |u_{x_i}|^{(p_i-2)\,\frac{\beta^{(\ell+1)}_N}{p_N}}\,dx\\
&+C\,\left(\frac{\|u\|_{L^\infty(B)}}{R'-R}\right)^{\frac{2}{p_N}\,\beta^{(\ell+1)}_N}\,\int_{B_{R'}} \sum_{i=j-1}^{N-1} |u_{x_i}|^{\frac{p_i-2}{p_N}\,\beta^{(\ell+1)}_N}\,dx.
\end{split}
\]
For the terms in the first sum of the right hand side, we use H\"older's inequality with the exponent 
\[
\frac{p_i}{p_i-2}\,\frac{p_N}{\beta^{(\ell+1)}_N},
\] 
and its conjugate\footnote{Observe that for $p_i=2$ there is no need of H\"older's inequality. If $p_i>2$, one can easily check that these exponents are larger than \(1\) by observing that 
\[
\beta_{N}^{(\ell+1)}\leq p_N\, q_{j-2}\leq p_N\, q_i,\qquad \mbox{ for } i\le j-2.
\]}. In the second sum, we use H\"older's inequality with the exponent\footnote{As before, there is no need of H\"older's inequality for $p_i=2$. For $p_i>2$, we rely on the fact that by definition
\[
\beta_{N}^{(\ell+1)}\leq \beta_{i}^{(\ell)}\,\frac{p_N}{p_i-2},\qquad \mbox{ for } j-1\le i\le N-1.
\] 
This justifies that the exponent is larger than \(1\).} 
\[
\frac{p_N}{p_i-2}\frac{\beta^{(\ell)}_i}{\beta^{(\ell+1)}_N}.
\] 
One gets
\[
\begin{split}
\int_{B_{R}}  | u_{x_N}|^{\beta^{(\ell+1)}_N}\,dx
&\le C\,(R')^N\,\left(\left(\frac{\|u\|_{L^\infty(B)}}{R'-R}\right)^{\beta^{(\ell+1)}_N}+\varepsilon_0\right)\\
&+C\,\left(\frac{\|u\|_{L^\infty(B)}}{R'-R}\right)^{2\,\frac{\beta^{(\ell+1)}_N}{p_N}}\,\sum_{i=1}^{j-2}\left(\int_{B_{R'}} |u_{x_i}|^{p_i}\,dx\right)^{\frac{p_i-2}{p_i}\,\frac{\beta^{(\ell+1)}_N}{p_N}}\\
&+C\,\left(\frac{\|u\|_{L^\infty(B)}}{R'-R}\right)^{2\,\frac{\beta^{(\ell+1)}_N}{p_N}}\,\sum_{i=j-1}^{N-1}\left(\int_{B_{R'}}  |u_{x_i}|^{\beta^{(\ell)}_i}\,dx\right)^{\frac{p_i-2}{p_N}\frac{\beta^{(\ell+1)}_N}{\beta^{(\ell)}_i}}.
\end{split}
\]
By using the induction assumption \eqref{last} to control the last term,   Lemma \ref{lm:below} and Proposition \ref{prop:Linfty} in order to control the other two, uniformly in $\varepsilon$, we get the desired estimate for $u_{x_N}$.
\par
We now assume that for some $m\in \{j,\dots,N\}$, we have 
\begin{equation}\label{sub-induction-assump}
\mbox{ for every }B_R\Subset B,\qquad \sum_{i=m}^N \int_{B_R} |u_{x_i}|^{\beta_{i}^{(\ell+1)}}\,dx\le C,\qquad \mbox{ with $C>0$ independent of $\varepsilon$},
\end{equation}
and prove that this entails
\[
\mbox{ for every }B_R\Subset B,\qquad \sum_{i=m-1}^N \int_{B_R} |u_{x_i}|^{\beta_{i}^{(\ell+1)}}\,dx\le C,\qquad \mbox{ with $C>0$ independent of $\varepsilon$}.
\]
Obviously, we only need to improve the control on the last component of the gradient, i.e. on $u_{_{x_{m-1}}}$. We still rely on Proposition \ref{prop:BFZ}, this time with the choices
\[
k=m-1\qquad \mbox{ and }\qquad p_{m-1}+2+\alpha=p_{m-1}\,\displaystyle\min\left\{q_{j-2},\, \min_{j-1\le i\le m-2}\frac{\beta^{(\ell)}_i}{p_i-2},\, \min_{m\le i\le N}\frac{\beta^{(\ell+1)}_i}{p_i-2}\right\}=\beta^{(\ell+1)}_{m-1}.
\]
The fact that \(\beta^{(\ell+1)}_{m-1} \geq \beta^{(0)}_{m-1}\geq p_{m-1}+2\) ensures that \(\alpha\geq 0\). Hence,  for every $B_R\Subset B_{R'}\Subset B,$ 
\[
\begin{split}
\int_{B_{R}}  | u_{x_{m-1}}|^{\beta^{(\ell+1)}_{m-1}}\,dx &\leq C\,(R')^N\,\left(\left(\frac{\|u\|_{L^\infty(B)}}{R'-R}\right)^{\beta^{(\ell+1)}_{m-1}}+\varepsilon_0\right)\\
&+C\,\left(\frac{\|u\|_{L^\infty(B)}}{R'-R}\right)^{\frac{2}{p_{m-1}}\,\beta^{(\ell+1)}_{m-1}}\,\int_{B_{R'}} \sum_{i=1}^{j-2} |u_{x_i}|^{(p_i-2)\,\frac{\beta^{(\ell+1)}_{m-1}}{p_{m-1}}}\,dx\\
&+C\,\left(\frac{\|u\|_{L^\infty(B)}}{R'-R}\right)^{\frac{2\,}{p_{m-1}}\beta^{(\ell+1)}_{m-1}}\,\int_{B_{R'}} \sum_{i=j-1}^{m-2} |u_{x_i}|^{(p_i-2)\,\frac{\beta^{(\ell+1)}_{m-1}}{p_{m-1}}}\,dx\\
&+C\,\left(\frac{\|u\|_{L^\infty(B)}}{R'-R}\right)^{\frac{2}{p_{m-1}}\,\beta^{(\ell+1)}_{m-1}}\,\int_{B_{R'}} \sum_{i=m}^{N} |u_{x_i}|^{(p_i-2)\,\frac{\beta^{(\ell+1)}_{m-1}}{p_{m-1}}}\,dx.
\end{split}
\]
We now proceed as above: we control the last term by using the induction assumption \eqref{sub-induction-assump} and the fact that if 
\[
\frac{\beta_{m-1}^{(\ell+1)}}{p_{m-1}}\,(p_i-2)\leq \beta_{i}^{(\ell+1)},\qquad \mbox{ if } m\le i\le N. 
\] 
The third term is estimated thanks to the induction assumption \eqref{last} and the inequality\footnote{This part of the discussion is void when $m=j$.} 
\[
\frac{\beta_{m-1}^{(\ell+1)}}{p_{m-1}}\,(p_i-2)\leq \beta_{i}^{(\ell)},\qquad \mbox{ if } j-1\leq i \leq m-2.
\] 
Finally,  on the two first terms, we use Lemma \ref{lm:below} and Proposition \ref{prop:Linfty}, and also that \footnote{This part of the discussion is void when $j=2$.}
\[
\frac{\beta_{m-1}^{(\ell+1)}}{p_{m-1}}\,(p_i-2)\leq p_i,\qquad \mbox{ if } 1\leq i \leq j-2.
\] 
This finally establishes that 
\[
\mbox{ for every }B_R\Subset B,\qquad \sum_{i=j-1}^N \int_{B_R} |u_{x_i}|^{\beta_{i}^{(\ell+1)}}\,dx\le C,
\]
with $C$ independent of $\varepsilon$. As already explained, this is enough to safely conclude the proof.
\end{proof}

\section{Proof of Theorem \ref{teo:lipschitz}}
\label{sec:7}

The cornerstones of the proof of Theorem \ref{teo:lipschitz} are the uniform $L^\infty$ estimate for the gradient of Proposition \ref{prop:a_priori_estimate} and the uniform higher integrability estimate of Proposition \ref{prop:pierre}. Indeed, by using Proposition \ref{prop:pierre} with the choice $q_0=\gamma$ (i.e. the exponent in \eqref{lipschitz}), we get that for every $B_{r_0}\Subset B$ with $r_0<1$
\[
\|\nabla u_\varepsilon\|_{L^\infty(B_{r_0})}\le C,
\]
with $C>0$ independent of $\varepsilon$. Observe that to infer that $C$ is independent of $\varepsilon$, we 
use Lemma \ref{lm:below} and Proposition \ref{prop:Linfty}.
Once we have this uniform estimate at our disposal, the Lipschitz regularity of $U$ follows with a standard covering argument, by taking into account that $u_\varepsilon$ converges to $U$ (see Lemma \ref{lm:convergence}). We refer to the proof of \cite[Theorem A]{BBJ} for  details.\hfill $\square$
\vskip.2cm
Once we have Theorem \ref{teo:lipschitz} at our disposal, we can prove a higher differentiability result {\it \`a la} Uhlenbeck. For the model case of the functional 
\[
\sum_{i=1}^N \frac{1}{p_i}\,\int |u_{x_i}|^{p_i}\,dx,
\]
the following result considerably improves \cite[Theorem 1.1]{BLPV}.
\begin{coro}
Let  $\mathbf{p}=(p_1,\dots,p_N)$ be such that $2\le p_1\le \dots\le p_N$.
Let $U\in W^{1,\mathbf{p}}_{\rm loc}(\Omega)$ be a local minimizer of $\mathfrak{F}_{\mathbf{p}}$ such that 
\[
U \in L^\infty_{\rm loc}(\Omega).
\]
Then 
\[
|U_{x_i}|^\frac{p_i-2}{2}\,U_{x_i} \in W^{1,2}_{\rm loc}(\Omega),\qquad \mbox{ for } i=1,\dots,N.
\]
\end{coro}
\begin{proof}
The proof is the same as the one in \cite[Proposition 3.2]{BLPV}. It is based on Nirenberg's method of incremental quotients, which aims at differentiating the equation
\[
\sum_{i=1}^N (|U_{x_i}|^{p_i-2}\,U_{x_i})=0,
\]
in a discrete sense. By proceeding as in \cite{BLPV}, we get for every $j=1,\dots,N$ and every pair of concentric balls $B_{r_0}\Subset B_{R_0}\Subset \Omega$
\[
\sum_{i=1}^N \int_{B_{r_0}} \left|\frac{\delta_{h\mathbf{e}_j }\left(|U_{x_i}|^\frac{p_i-2}{2}\,U_{x_i}\right)}{|h|^\frac{s_j+1}{2}}\right|^2\,dx\le \frac{C}{(R_0-r_0)^2}\,\sum_{i=1}^N \left(\int_{B_{R_0}} |U_{x_i}|^{p_i}\,dx\right)^\frac{p_i-2}{p_i}\,\left(\int_{B_{R_0}} \left|\frac{\delta_{h\mathbf{e}_j} U}{|h|^\frac{s_j+1}{2}}\right|^{p_i}\right)^\frac{2}{p_i},
\]
see \cite[equation (3.6)]{BLPV}. By using that $\nabla U\in L^\infty_{\rm loc}$, we can choose
\[
s_j=1\qquad \mbox{ so that }\qquad \frac{s_j+1}{2}=1,
\]
to control the last term on the right-hand side and obtain an estimate on 
\[
\sum_{i=1}^N \int_{B_{r_0}} \left|\frac{\delta_{h\mathbf{e}_j }\left(|U_{x_i}|^\frac{p_i-2}{2}\,U_{x_i}\right)}{|h|}\right|^2\,dx,\qquad j=1,\dots,N,
\]
which is uniform in $|h|\ll 1$. By appealing to the difference quotient characterization of Sobolev spaces, we get the conclusion.
\end{proof}

\appendix

\section{Calculus lemmas}
In this section, we separately present some proofs on the elementary facts for the sequences needed in the proof of our main result.
\subsection{Tools for the Lipschitz estimate}
In what follows, we denote as usual
\[
2^*=\frac{2\,N}{N-2},\qquad \mbox{ for }N\ge 3.
\]
\begin{lm}
\label{lm:madonne}
Let $2\le p_1\le p_N$. We define
\[
j_0=\min\bigg\{j\in\mathbb{N}\,:\, j>\log_2\bigg(\frac{N-2}{2}\,(p_N-2)-\frac{N}{2}\,(p_1-2)\bigg)-2\bigg\},
\]
\[
j_1=\min\bigg\{j\in\mathbb{N}\,:\, j>\log_2\bigg((N-2)\,(p_N-2)-N\,\left(p_1-2\right)\bigg)-2\bigg\}.
\]
and $J=1+\max\{j_0,j_1\}$.
We set
\[
\gamma_j=2^{j+2}+p_N-2,
\]
and
\[
\tau_j=\frac{\gamma_{j-1}}{\gamma_j}\,\frac{\dfrac{2^*}{2}\,(\gamma_j+p_1-p_N)-\gamma_j}{\dfrac{2^*}{2}\,(\gamma_j+p_1-p_N)-\gamma_{j-1}}.
\]
Then there exist two constants $0<C_1<C_2<1$ depending on $N, p_1$ and $p_N$ such that
\[
C_1\le\frac{(1-\tau_j)\,\gamma_j}{\gamma_j+p_1-p_N}\le C_2,\qquad \mbox{ for every } j\ge J.
\]
\end{lm}
\begin{proof}
It is easily seen that the sequence $\{\gamma_j\}_{j\in\mathbb{N}}$ is increasing. Moreover, by definition of $j_0$, we have that 
\[
\gamma_j<\frac{2^*}{2}\,(\gamma_j+p_1-p_N),\qquad \mbox{ for every }j\ge j_0,
\]
thus $\tau_j$ is well-defined and positive for $j\ge j_0$. The definition of $\tau_j$ entails that
\begin{equation}
\label{interpola}
\frac{1}{\gamma_j}=\frac{\tau_j}{\gamma_{j-1}}+\frac{1-\tau_j}{\dfrac{2^*}{2}\,(\gamma_j+p_1-p_N)}.
\end{equation}
thus the previous discussion mplies that 
\[
0<\tau_j<1,\qquad \mbox{ for every } j\ge j_0.
\]
Then the proof is by direct computation: we have
\[
\begin{split}
\frac{(1-\tau_j)\,\gamma_j}{\gamma_j+p_1-p_N}&=\left(1-\frac{\gamma_{j-1}}{\gamma_j}\,\frac{\dfrac{2^*}{2}\,(\gamma_j+p_1-p_N)-\gamma_j}{\dfrac{2^*}{2}\,(\gamma_j+p_1-p_N)-\gamma_{j-1}}\right)\,\frac{\gamma_j}{\gamma_j+p_1-p_N}\\
&=\frac{\gamma_j\,\dfrac{2^*}{2}\,(\gamma_j+p_1-p_N)-\gamma_{j-1}\,\dfrac{2^*}{2}\,(\gamma_j+p_1-p_N)}{\dfrac{2^*}{2}\,(\gamma_j+p_1-p_N)-\gamma_{j-1}}\,\frac{1}{\gamma_j+p_1-p_N}\\
&=\frac{2^*}{2}\,\frac{\gamma_j-\gamma_{j-1}}{\dfrac{2^*}{2}\,(\gamma_j+p_1-p_N)-\gamma_{j-1}}\\
&=\frac{2^*}{2}\,\frac{2^{j+2}-2^{j+1}}{\dfrac{2^*}{2}\,(2^{j+2}+p_1-2)-p_N-2^{j+1}+2}\\
&=\frac{2^*}{4}\,\frac{2^{j+2}}{\left(\dfrac{2^*-1}{2}\right)\,2^{j+2}+\dfrac{2^*}{2}\,(p_1-2)-(p_N-2)}.
\end{split}
\]
We have to distinguish two cases: if 
\[
p_N<\frac{N}{N-2}\,(p_1-2)+2,
\]
then the function
\begin{equation}
\label{funzioncina}
t\mapsto \frac{t}{\left(\dfrac{2^*-1}{2}\right)\,t+\dfrac{2^*}{2}\,(p_1-2)-(p_N-2)},
\end{equation}
is well-defined for every $t>0$ and monotonically {\it increasing}. We have in this case
\[
\begin{split}
0<\frac{(1-\tau_{j_0})\,\gamma_{j_0}}{\gamma_{j_0}+p_1-p_N}&\le \frac{(1-\tau_j)\,\gamma_j}{\gamma_j+p_1-p_N}\\&<\lim_{j\to\infty} \frac{2^*}{4}\,\frac{2^{j+2}}{\left(\dfrac{2^*-1}{2}\right)\,2^{j+2}+\dfrac{2^*}{2}\,(p_1-2)-(p_N-2)}=\frac{1}{2}\,\frac{2^*}{2^*-1}<1,
\end{split}
\]
for every $j\ge j_0$.
\par
On the other hand, if 
\begin{equation}\label{eq1525}
p_N\ge \frac{N}{N-2}\,(p_1-2)+2,
\end{equation}
then the function \eqref{funzioncina} is well-defined and monotonically {\it decreasing} for 
\begin{equation}\label{eq1529}
t>\frac{2}{2^*-1}\,\left((p_N-2)-\dfrac{2^*}{2}\,(p_1-2)\right).
\end{equation}
Thus we now obtain
\[
\frac{1}{2}\,\frac{2^*}{2^*-1}\le \frac{(1-\tau_j)\,\gamma_j}{\gamma_j+p_1-p_N}\le \frac{2^*}{4}\,\frac{2^{j_1+2}}{\left(\dfrac{2^*-1}{2}\right)\,2^{j_1+2}+\dfrac{2^*}{2}\,(p_1-2)-(p_N-2)},
\]
for every $j\ge j_1$. Observe that the choice of $j_1$ assures firstly that \(t=2^{j+2}\) satisfies \eqref{eq1529} whenever \(j\geq j_1\) (here we  use \eqref{eq1525}) and secondly, that the right-hand side above is strictly smaller than $1$.
\end{proof}

\begin{lm}
\label{lm:putain}
With the notation of Lemma \ref{lm:madonne}, we define the sequence \(\{\varepsilon_j\}_{j\geq j_0}\) by
\begin{equation}
\label{madonnina}
1+\varepsilon_j=\tau_j\,\left(\frac{\gamma_j+p_1-p_N}{(1-\tau_j)\,\gamma_j}\right)',\qquad \mbox{ for } j\ge j_0.
\end{equation}
Then 
\[
\varepsilon_j\sim \frac{N}{4}\,\frac{p_N-p_1}{2}\,\frac{1}{2^j},\qquad \mbox{ for } j\to \infty.
\]
In particular, we have
\[
\lim_{n\to\infty} \prod_{i=j_0}^n (1+\varepsilon_j)<+\infty.
\]
\end{lm}
\begin{proof}
We start by computing explicitly the conjugate exponent appearing in \eqref{madonnina}. We have
\[
\begin{split}
\left(\frac{\gamma_j+p_1-p_N}{(1-\tau_j)\,\gamma_j}\right)'&=\left(\frac{1}{1-\tau_j}\,\left(1+\frac{p_1-p_N}{\gamma_j}\right)\right)'\\
&=\frac{\dfrac{1}{1-\tau_j}\,\left(1+\dfrac{p_1-p_N}{\gamma_j}\right)}{\dfrac{1}{1-\tau_j}\,\left(1+\dfrac{p_1-p_N}{\gamma_j}\right)-1}=\frac{1-\dfrac{p_N-p_1}{\gamma_j}}{\tau_j-\dfrac{p_N-p_1}{\gamma_j}}.
\end{split}
\]
Thus we have
\[
\varepsilon_j=\tau_j\,\left(\frac{\gamma_j+p_1-p_N}{(1-\tau_j)\,\gamma_j}\right)'-1=\frac{\tau_j\,\left(1-\dfrac{p_N-p_1}{\gamma_j}\right)}{\tau_j-\dfrac{p_N-p_1}{\gamma_j}}-1=\frac{p_N-p_1}{\gamma_j}\,\frac{1-\tau_j}{\tau_j-\dfrac{p_N-p_1}{\gamma_j}}.
\]
We now observe that 
\[
\frac{p_N-p_1}{\gamma_j}\,\frac{1-\tau_j}{\tau_j-\dfrac{p_N-p_1}{\gamma_j}}\sim \frac{p_N-p_1}{\gamma_j}\,\frac{1-\tau_j}{\tau_j},\qquad \mbox{ for } j\to \infty.
\]
Moreover, by using the definitions of $\gamma_j$ and $\tau_j$, we have
\[
\frac{1-\tau_j}{\tau_j}=\frac{1}{\tau_j}-1\sim \frac{N}{2},\qquad \mbox{ for } j\to\infty,
\]
which implies that
\[
\varepsilon_j\sim \frac{N}{2}\,\frac{p_N-p_1}{\gamma_j},\qquad \mbox{ for } j\to\infty.
\]
By observing that $\gamma_j\sim 2^{j+2}$, we get the desired conclusion.
\par
In order to prove the last part, it is enough to notice that
\[
\prod_{j=j_0}^n (1+\varepsilon_j)=\exp\left(\sum_{j=j_0}^n\log (1+\varepsilon_j)\right)\qquad \mbox{ and }\qquad \log (1+\varepsilon_j)\sim \varepsilon_j,\qquad \mbox{ for } j\to\infty.
\]
By using the first part of the proof and the definition of $\gamma_j$, we see that
\[
\lim_{n\to\infty}\sum_{j=j_0}^n \varepsilon_j<+\infty.
\]
This concludes the proof.
\end{proof}

\subsection{Tools for the higher integrability}

In this section, we present some properties of the vector-valued sequence $\{(\beta^{\ell}_{j-1},\dots,\beta^{\ell}_N)\}$ which were needed in the proof of Proposition \ref{prop:pierre}, in order to complete the inductive step.
We use the same notation as before: in particular, we fix \(2\leq q_0<+\infty\) and set
\[
q_j=\min \left\{\left(\frac{p_j}{2}\right)', q_0\right\},\qquad j=1,\dots,N.
\]
Then for a fixed index $j\in\{2,\dots,N\}$, we define
\begin{equation}
\label{boh_appendix}
\left\{\begin{array}{ccl}
\beta_{j-1}^{(0)}&=&p_{j-1} \min \left\{q_{j-2},\, q_{j-1}\,q_N \right\},\\
&&\\
\beta_{i}^{(0)}&=& p_i\, q_{j-1},\qquad \mbox{ for } i=j,\dots,N,\\
\end{array}
\right.
\end{equation}
and by a recursive scheme
\begin{equation}\label{eq-def-betai}
\left\{\begin{array}{ccl}
\beta_N^{(\ell+1)}&=&p_N\,\displaystyle\min\left\{q_{j-2},\, \min_{j-1\le k\le N-1}\frac{\beta^{(\ell)}_k}{p_k-2}\right\}\\
\beta_{N-1}^{(\ell+1)}&=&p_{N-1}\,\displaystyle\min\left\{q_{j-2},\, \min_{j-1\le k\le N-2}\frac{\beta^{(\ell)}_k}{p_k-2},\, \frac{\beta^{(\ell+1)}_N}{p_N-2}\right\}\\
\beta_{N-2}^{(\ell+1)}&=&p_{N-2}\,\displaystyle\min\left\{q_{j-2},\, \min_{j-1\le k\le N-3}\frac{\beta^{(\ell)}_k}{p_k-2},\, \min_{N-1\le k\le N}\frac{\beta^{(\ell+1)}_k}{p_k-2}\right\}\\
\vdots & = &\vdots\\
\beta_{j-1}^{(\ell+1)}&=&p_{j-1}\,\displaystyle\min\left\{q_{j-2},\, \min_{j\le k\le N}\frac{\beta^{(\ell+1)}_k}{p_k-2}\right\}
\end{array}
\right.
\end{equation}

\begin{lm}
\label{lm-beta-increasing}
For every \(i\in \{j-1, \dots, N\}\), the sequence \(\{\beta_{i}^{(\ell)}\}_{\ell\in \mathbb{N}}\) is nondecreasing.
\end{lm}
\begin{proof}
We proceed again by induction on \(\ell\).
\vskip.2cm\noindent
{\bf Initialization step}. We first need to prove that
\begin{equation}
\label{step0}
\beta_{i}^{(1)}\geq \beta_{i}^{(0)},\qquad \mbox{ for }  i=j-1,\dots,N.
\end{equation}
We establish this by downward induction on \(i=N, \dots, j-1\). Indeed, for \(i=N\), one has by definition
\[
\beta_{N}^{(1)} = p_N\min \left\{q_{j-2},\, \min_{j-1\leq k\le N-1} \frac{\beta_{k}^{(0)}}{p_k-2}\right\}. 
\]
By using the definition of $\beta^{(0)}_k$, the previous is the same as
\[
\begin{split}
\beta_{N}^{(1)} 
&=p_N\,\min \Big\{q_{j-2},\,q_{j-1}\,\min \{q_{j-2},\, q_N \,q_{j-1}\},\,q_{j-1}\,\min_{j\le k\le N-1} q_k\Big\}. 
\end{split}
\]
By recalling that $q_{j-2}\geq \dots \geq q_N$, this gives
\[
\beta_{N}^{(1)} = p_N\,\min \Big\{q_{j-2},\,q_{j-1}\,\min\{q_{j-2},\, q_N\, q_{j-1}\},\,q_{j-1}\,q_{N-1}\Big\}  \geq  p_N\, \min \{q_{j-2},\,q_{j-1}\,q_N\} . 
\]
Hence, 
\[
\beta_{N}^{(0)}=p_N\,q_{j-1} \leq p_N \min\{ q_{j-2},\, q_{j-1}\,q_N\} = \beta_{N}^{(1)}.
\] 
This proves \eqref{step0} for \(i=N\). 
\par
We now assume that for some \(i\in \{j-1, \dots, N\}\), property \eqref{step0} holds for every \(k\in \{i+1, \dots, N\}\). One proceeds to prove that \eqref{step0} holds for $i$, as well. By definition of \(\beta_{i}^{(1)}\), 
\[
\beta_{i}^{(1)}=p_i\, \min \left\{q_{j-2},\,\min_{j-1\leq k\le i-1} \frac{\beta_{k}^{(0)}}{p_k-2},\, \min_{i+1\le k\le N} \frac{\beta_{k}^{(1)}}{p_k-2} \right\}.
\]
By the induction assumption, \(\beta_{k}^{(1)}\geq \beta_{k}^{(0)}\) for \(k\geq i+1\) and thus
\begin{align*}
\beta_{i}^{(1)}&\geq p_i\, \min \left\{q_{j-2},\,\min_{j-1\leq k\le i-1} \frac{\beta_{k}^{(0)}}{p_k-2} ,\, \min_{i+1\le k\le N} \frac{\beta_{k}^{(0)}}{p_k-2} \right\}\\
&=p_i\,\min \left\{q_{j-2},\, \frac{\beta_{j-1}^{(0)}}{p_{j-1}-2},\, \min_{\substack{j\le k\le N,\\ k\not=i}} \frac{\beta_{k}^{(0)}}{p_k-2}\right\}.
\end{align*}
By definition of \(\beta_{i}^{(0)} \), this gives
\[
\begin{split}
\beta_{i}^{(1)}&\geq p_i\,\min \Big\{q_{j-2},\, q_{j-1}\,\min\{q_{j-2},\, q_N\, q_{j-1}\} ,\, q_{j-1}\,\min_{\substack{j\le k\le N,\\ k\not=i}} q_k\Big\}\\
& = p_i\,\min \{q_{j-2},\, q_{j-1}\,q_N\}.
\end{split}
\]
When \(i\ge j\), this implies \(\beta_{i}^{(1)}\geq p_i\,q_{j-1}=\beta_{i}^{(0)}\), while when \(i=j-1\), one has
\[
\beta_{j-1}^{(1)}\geq p_{j-1}\,\min\{q_{j-2},\, q_{N}\,q_{j-1}\}=\beta_{j-1}^{(0)}.
\]
We have thus proved \eqref{step0} for \(i\), which completes the proof.
\vskip.2cm\noindent
{\bf Inductive step}. We now assume that for an index  \(\ell\geq 1\), we have
\begin{equation}
\label{17072018}
\beta^{(\ell)}_i\ge \beta^{(\ell-1)}_i,\qquad \mbox{ for } i=j-1,\dots,N. 
\end{equation}
We need to prove that this entails
\[
\beta^{(\ell+1)}_i\ge \beta^{(\ell)}_i,\qquad \mbox{ for } i=j-1,\dots,N,
\]
as well. 
\par
We rely  again on a downward induction on \(i=N, \dots, j-1\). Indeed, for \(i=N\), we can use \eqref{17072018}, which gives
\[
\beta_{N}^{(\ell+1)}=p_N \min\left\{q_{j-2},\,\min_{j-1\leq k\le N-1} \frac{\beta_{k}^{(\ell)}}{p_k-2}\right\} \geq p_N \min\left\{ q_{j-2},\,\min_{j-1\leq k\le N-1} \frac{\beta_{k}^{(\ell-1)}}{p_k-2}\right\}=\beta_{N}^{(\ell)}.
\]
We now assume that for some \(i\geq j-1\), one has 
\[
\beta_{k}^{(\ell+1)}\geq \beta_{k}^{(\ell)},\qquad \mbox{ for every } k\in \{i+1, \dots, N\}.
\] 
Then
\[
\begin{split}
\beta_{i}^{(\ell+1)}&=p_i\, \min \left\{q_{j-2},\,\min_{j-1\leq k\le i-1} \frac{\beta_{k}^{(\ell)}}{p_k-2},\, \min_{i+1\le k\le N} \frac{\beta_{k}^{(\ell+1)}}{p_k-2} \right\}\\
& \geq p_i\, \min \left\{q_{j-2},\,\min_{j-1\leq k\le i-1} \frac{\beta_{k}^{(\ell)}}{p_k-2} ,\, \min_{i+1\le k\le N} \frac{\beta_{k}^{(\ell)}}{p_k-2} \right\}.
\end{split}
\]
Relying now on the induction assumption \eqref{17072018}, one gets
\[
\beta_{i}^{(\ell+1)} \geq p_i \min \left\{q_{j-2},\,\min_{j-1\leq k\le i-1} \frac{\beta_{k}^{(\ell-1)}}{p_k-2} , \min_{i+1\le k\le N} \frac{\beta_{k}^{(\ell)}}{p_k-2} \right\}=\beta_{i}^{(\ell)}.
\]
This completes the proof. 
\end{proof}

\begin{lm}
\label{lm:limitbetai}
With the notation of the previous lemma, there exists \(\ell_0\in \mathbb{N}\) such that for every \(\ell\geq \ell_0\), one has
\[
\beta_{i}^{(\ell)}=p_i\, q_{j-2},\qquad \mbox{ for every } i=j-1,\dots,N.
\]
\end{lm}
\begin{proof}
By using the monotonicity proved in the previous lemma, we get in particular for \(i=j-1,\dots,N\),
\begin{equation}
\label{limitozzo}
\begin{split}
\beta^{(\ell+1)}_i&=p_{i}\,\displaystyle\min\left\{q_{j-2},\, \min_{j-1\le k\le i-1}\frac{\beta^{(\ell)}_k}{p_k-2},\, \min_{i+1\le k\le N}\frac{\beta^{(\ell+1)}_k}{p_k-2}\right\}\\
&\ge p_{i}\,\displaystyle\min\left\{q_{j-2},\, \min_{\substack{j-1\le k\le N,\\ k\not=i}} \frac{\beta^{(\ell)}_k}{p_k-2}\right\}\ge p_{i}\,\displaystyle\min\left\{q_{j-2},\, \min_{j-1\le k\le N} \frac{\beta^{(\ell)}_k}{p_k-2}\right\}.
\end{split}
\end{equation}
Dividing by \(p_i\) and observing that 
\[
\frac{\beta^{(\ell)}_k}{p_k-2} \geq  \frac{\beta^{(\ell)}_k}{p_k}q_k\geq \frac{\beta^{(\ell)}_k}{p_k}\,q_N,
\] 
one deduces that
\[
\frac{\beta^{(\ell+1)}_i}{p_i} \geq  \displaystyle\min\left\{q_{j-2},\, q_N\,\min_{j-1\le k\le N} \frac{\beta^{(\ell)}_k}{p_k}\right\}.
\]
Since this is true for every \(i=j-1,\dots,N\), this implies 
\begin{equation}\label{eq1799}
\delta^{(\ell+1)} \geq \displaystyle\min\left\{q_{j-2},\, q_N\,\delta^{(\ell)}\right\},\qquad \mbox{ where } \delta^{(\ell)}=\min_{j-1\le k\le N}\frac{\beta^{(\ell)}_k}{p_k}.
\end{equation}
The monotonicity of each sequence \(\{\beta^{(\ell)}_k\}_{\ell\in\mathbb{N}}\) entails the monotonicity of \(\{\delta^{(\ell)}\}_{\ell\in\mathbb{N}}\). We claim that there exists \(\ell_0\in \mathbb{N}\) such that one has 
\begin{equation}\label{eqcl-1801}
q_N\, \delta^{(\ell)} \geq q_{j-2},\qquad \mbox{ for every }\ell\geq \ell_0.
\end{equation} 
Indeed, assume by contradiction that \(q_N\, \delta^{(\ell)} < q_{j-2}\) for every $\ell\ge 0$. Then it follows from \eqref{eq1799} that 
\[
\delta^{(\ell+1)} \geq q_N\,\delta^{(\ell)},\qquad \mbox{ for every } \ell\ge 0.
\]
which implies in turn that \(\delta^{(\ell)}\nearrow+\infty \), as $\ell\to \infty$. This contradicts that \(q_N \delta^{(\ell)} < q_{j-2}\) for every \(\ell\geq 0\). Hence, the claim \eqref{eqcl-1801} is established.
By \eqref{eq1799} again, this implies that 
\[
\frac{\beta_i^{(\ell+1)}}{p_i}\ge \min_{j-1\le k\le N}\frac{\beta^{(\ell+1)}_k}{p_k}=\delta^{(\ell+1)} \geq q_{j-2},\qquad  \mbox{ for } i=j-1,\dots,N,\mbox{ for every }\ell\ge \ell_0-1.
\]
Since the opposite estimate on $\beta_i^{(\ell+1)}/p_i$ is a consequence of the definition of \(\beta_{i}^{(\ell+1)}\), we obtain the desired conclusion.
\end{proof}

\end{document}